\documentclass[11pt]{article}

\usepackage{amssymb}
\usepackage{amsthm}
\usepackage{amsmath}
\usepackage{amscd}
\usepackage{graphicx}
\usepackage{hyperref}

\allowdisplaybreaks
    
\newcommand{\R}{\mathbb{R}}

\newcommand{\N}{\mathbb{N}}
\newcommand{\Sph}{\mathbb{S}}

\newcommand{\eps}{\varepsilon}

\newcommand{\dif}{\mathrm{d}}

\newcommand{\vol}{\mathrm{Vol}}

\newcommand{\coan}{C^{0,\alpha}_{\nu - 2}}

\newcommand{\ctan}{C^{2,\alpha}_\nu}

\newcommand{\ckan}{C^{k,\alpha}_\nu}

\newcommand{\riem}{\mathrm{Rm}}
\newcommand{\ric}{\mathrm{Ric}}

\newcommand{\fullapsol}{\tilde \Sigma_r(\Gamma, \tau, W, \Xi)}
\newcommand{\apsol}{\tilde \Sigma_r}

\DeclareMathOperator{\arccosh}{arccosh}

\newtheorem{thm}{Theorem}
\newtheorem{lemma}[thm]{Lemma}

\newtheorem{prop}[thm]{Proposition}
\newtheorem*{nonumthm}{Theorem}

\theoremstyle{definition}
\newtheorem{defn}[thm]{Definition}

\newtheoremstyle{rmk}{5pt}{5pt}{}{}{\scshape}{:}{.5em}{}
\theoremstyle{rmk}

\newcommand{\mylabel}
	{\label}
\begin{document}

\title{CMC Surfaces in Riemannian Manifolds \\
Condensing to a Compact Network of Curves} 

\author{Adrian Butscher \\  Stanford University \\ Department of Mathematics \\ email: butscher@stanford.edu}

\maketitle

\begin{abstract}
A sequence of constant mean curvature surfaces $\Sigma_j$ with mean curvature $H_j \rightarrow \infty$ in a three-dimensional manifold $M$ condenses to a compact and connected graph $\Gamma$ consisting of a finite union of curves if $\Sigma_j$ is contained in a tubular neighbourhood of $\Gamma$ of size $\mathcal O(1/H_j)$ for every $j \in \N$.  This paper gives sufficient conditions on $\Gamma$ for the existence of a sequence of compact, embedded constant mean curvature surfaces condensing to $\Gamma$.  The conditions are: each curve in $\gamma$ is a critical point of a functional involving the scalar curvature of $M$ along $\gamma$; and each curve must satisfy certain regularity, non-degeneracy and boundary conditions.  When these conditions are satisfied, the surfaces $\Sigma_j$ can be  constructed by gluing together small spheres of radius $2/H_j$ positioned end-to-end along the edges of $\Gamma$.
\end{abstract}

\renewcommand{\baselinestretch}{1.25}
\normalsize

\section{Introduction}
\label{sec:intro}

\paragraph{Background.}  A constant mean curvature (CMC) hypersurface $\Sigma$ contained in an ambient (for simplicity three-dimensional) compact Riemannian manifold $M$ has the property that its mean curvature with respect to the induced metric is constant.  This property ensures that area of $\Sigma$ is a critical value of the area functional for surfaces of $M$ subject to an enclosed-volume constraint.  The study of CMC surfaces in the classical setting of $M = \R^3$ is a well established field of Riemannian geometry and the literature concerning the construction and properties of such surfaces is enormous.  Indeed, it is known that there is a very great flexibility in the construction of CMC surfaces and that also a number of interesting rigidity results hold as well.

Examples of such rigidity results are as follows.  In the particular case of finite-topology surfaces, Meeks \cite{meeks} proved that any end of a complete Alexandrov-embedded CMC surface in $\R^3$ is cylindrically bounded; while Korevaar, Kusner and Solomon \cite{kks} improved this by showing that any such end converges exponentially to one end of a Delaunay surface.  Furthermore, the possible directions of the axes of these ends are also subject to limitations, as well as the flexibility to change these directions within  the moduli space of all such surfaces, see \cite{gbkkrs,gks,gks2}.  These limitations are phrased in terms of a certain \emph{flux} that was discovered by Kusner. The flux is a vector associated to any closed loop in a CMC surface; it is constant under deformations of this loop, and in fact only depends on the homology class of this loop in the surface. There is a flux associated to a simple positively oriented loop around each asymptotically Delaunay end, which depends only the direction of the axis and the neck-size of the limiting Delaunay surface. The homological invariance also shows that the sum over all ends of these limiting fluxes must vanish, which is a global balancing condition for the entire CMC surface. One immediate consequence is the non-existence of a complete Alexandrov-embedded CMC surface in $\R^3$ with only one end. Finally, Korevaar and Kusner \cite{korevaarkusner1} have a very general structure theorem for these types of CMC surfaces which states that the ends are as described above while the remainder of the surface is contained in a tubular neighbourhood of a graph consisting of line segments and a balancing formula holds at the nodes of the graph. 

The corresponding picture in general Riemannian manifolds is considerably less well-developed.  Rather than focus on special cases (such as when $M$ is the sphere or hyperbolic space, where there are many interesting results) we will work in an unspecified Riemannian manifold but consider small CMC surfaces of high mean curvature.  In this setting, Rosenberg \cite{rosenberg} has shown that if $\Sigma$ is a closed CMC surface in an arbitrary (compact) $3$-manifold $M$, with sufficiently large mean curvature $H$, then $M \setminus \Sigma$ has two components, and the inradius at any point in one of these components is bounded above by $C/H$. In other words, $\Sigma$ is localized in a small tubular neighbourhood of $\Gamma$.  This phenomenon is captured via the following definition.

\begin{defn} A sequence of constant mean curvature surfaces $\Sigma_j$ in  $M$ with mean curvature $H_j \rightarrow \infty$ \emph{condenses} to  $\Gamma$ if $\Sigma_j$ is contained in a tubular neighbourhood of $\Gamma$ of size $\mathcal O(1/H_j)$ for every $j \in \N$.
\end{defn}

\noindent We are led to the following two central and very natural questions.  

\begin{itemize}
	\item What are the possible condensation sets $\Gamma$ for condensing sequences of CMC surfaces $\Sigma_j$?
	
	\item If $\Sigma_j $ condenses to $ \Gamma$ then what can be said about the nature of $\Sigma_j$ for large $j$?
\end{itemize}

\noindent As an example of the kind of answers we hope to provide, let us consider a sequence of dilations of a complete CMC surface with $k$ asymptotically Delaunay ends in $\R^3$, where the dilation parameter goes to zero.  Such a sequence condenses onto a one-dimensional set, here a union of half-lines meeting  at a point where the unit vectors defining the half-lines satisfy the flux balancing condition described above (which specifically means that the sum of these unit vectors, weighted by a factor relating to the Delaunay parameter of the corresponding end, vanishes).  

An obvious hope is that this Euclidean example represents the local behaviour of condensing sequences of CMC surfaces in general Riemannian manifold; i.e.\ that in a Riemannian manifold $M$, a condensation set $\Gamma$ is some sort of network of geodesics, where each edge of this geodesic network to have a `weighting' which carries information about the Delaunay parameters of the CMC tubular piece which converges to that edge, and that the nodes of this network are balanced.   There is in fact some evidence that this is true, in the form of a result of Mazzeo and Pacard \cite{mazzeopacardtubes}.  They show that if $\Gamma$ is any closed, non-degenerate geodesic in $M$, then most geodesic tubes of sufficiently small radius about $\Gamma$ can be perturbed into CMC surfaces with large $H$ (thus providing an example of a sequence of CMC surfaces condensing to $\Gamma$.  Furthermore, they show that if $\Sigma_j$ condenses onto a smooth one-dimensional manifold $\Gamma$ and the pointwise ratio of the norm of the second fundamental form to the mean curvature is bounded above independently of $j$, then $\Gamma$ must be a geodesic.

Despite this compelling evidence, however, it seems that the actual characterization of condensing sequences of CMC surfaces in $M$ is potentially far more complex and interesting.  To understand the reason why this is so, it is helpful to review the first result about CMC surfaces of high mean curvature.  This result is due to Ye \cite{ye} in the early 1990s and relates to sequences of CMC surfaces collapsing to a single point, so that $\Gamma = \{ p \}$ for some $p$ in $M$.  Ye's result has two components.  First, he proves that if $p$ is a non-degenerate critical point of the scalar curvature of $M$, then geodesic spheres around $p$ with small radius may be perturbed into CMC surfaces with large $H$.  Second, he also proves conversely that if a sequence of CMC spheres with $H \to \infty$ forms a local foliation with uniformly bounded pointwise ratio of the norm of the second fundamental form to the mean curvature, then this sequence converges to a critical point of the scalar curvature.  

The conclusion to be drawn from Ye's work is that the scalar curvature of $M$ may be involved in the behaviour of condensing sequences of CMC surfaces in some way.  The question now becomes if the scalar curvature of $M$ influences the behaviour of sequences of CMC surfaces condensing to a one-dimensional variety.    Given Mazzeo and Pacard's result \cite{mazzeopacardtubes} cited above, we should only expect such an influence if $\Gamma$ is a non-smooth variety or else  the pointwise ratio of the norm of the second fundamental form to the mean curvature of $\Sigma_j$ diverges on at least one sequence of points $p_j \in \Sigma_j$.  

Butscher and Mazzeo \cite{memazzeo} have recently shown that it is indeed possible to construct a condensing sequence of CMC surfaces where the condensation set is contrary to what one would expect in the Euclidean case.  Namely, they construct surfaces that locally resemble a Delaunay surface but are either compact or have one end.   It should be noted that their construction depends critically on the nature of $M$ --- they assume that $M$ is cylindrically symmetric (about a geodesic $\gamma$) and very rapidly asymptotically Euclidean.  They use the \emph{gluing technique} 
\footnote{The gluing technique is a well-known technique used in geometric analysis.  It goes back to the work of Kapouleas \cite{kapouleas7,kapouleas6} in the context of CMC surfaces and has been further developed by many other researchers. See \cite{mazzeosurvey,pacardsurvey} for surveys about the current  state of this approach}
to construct their surfaces: a collection of  $\mathcal O(L/r)$ small geodesic spheres of radius $r$ are positioned end-to-end along $\gamma$ and glued to each other using re-scaled, embedded catenoids to produce a compact, almost-CMC surface of length $\mathcal O(L)$.  A semi-infinite embedded Delaunay surface can also be glued to one of the terminal spheres of this configuration to produce a one-ended almost-CMC surface.  These surfaces are then perturbed into true CMC surfaces by solving a partial differential equation.  The sizes of the catenoidal necks employed in the construction vary along the length of $\gamma$ in a manner governed by a flux formula that  implies that the difference between successive neck-sizes can be expressed in terms of the gradient of the scalar curvature along the axis connecting these two necks.  Moreover, these neck sizes are quite small compared to the radii of the spheres --- the pointwise ratio of the norm of the second fundamental form to the mean curvature diverges precisely in the neck regions of these surfaces.  

The condensation sets in Butscher and Mazzeo's work above are very simple: a geodesic segment on the one hand and a geodesic ray on the other.   The reason for this is the highly restrictive symmetry conditions satisfied by $M$, which in particular forces the gradient of the scalar curvature to point along $\gamma$.  One might thus hope for more interesting behaviour and more radical departures from the Euclidean picture if these symmetry conditions are weakened.  In this paper, we consider a completely general (though compact) ambient manifold $M$ containing a union of curves $\Gamma$.  We identify a set of sufficient conditions on the curves in $\Gamma$ and on the endpoints of these curves that must hold in order for the construction of a sequence of embedded CMC surfaces condensing to $\Gamma$ to be possible via the gluing technique.  These conditions, expressed in greater detail below, show that the scalar curvature affects the curves in $\Gamma$ as well as the balancing formul\ae\ satisfied at the endpoints of these curves.  In particular, these curves are in general neither geodesics nor integral curves of the gradient of the scalar curvature.

\paragraph{Results.}  Consider a graph $\Gamma$ in $M$ with edges $\mathcal E := \{ \gamma_1, \ldots, \gamma_E\}$  and vertices $\mathcal V := \{ p_1, \ldots, p_N\}$.  We will assume that the vertices are found amongst the endpoints of the edges and never amongst the interior points of the edges.  Furthermore, if there are two edges $\gamma_e$ and $\gamma_{e'}$ emanating from $p_i$ then the one-sided tangent vectors of $\gamma_e$ and $\gamma_{e'}$ at $p_i$ are not co-linear.   We will also assume that each $\gamma_e$ is parametrized by arc-length and we will also use the mapping $ t \mapsto \gamma_e(t)$ to denote this parametrization.  We now introduce four \emph{existence conditions} on the curves in $\Gamma$.

\begin{enumerate}
	
	\item Each $\gamma_e \in \mathcal E$ is smooth with geodesic curvature uniformly bounded independently of $r$.

	\item For each $\gamma \in \mathcal E$ and all sufficiently small $r$, there is a smooth, positive function $f: \gamma \rightarrow \R$ so that the pair $(\gamma, f)$ satisfies the equation
	\begin{subequations}
	\begin{equation}
		\label{eqn:curveeqn}
		\nabla_{\dot \gamma} \big( f \dot \gamma \big) := \Omega r^2 \nabla R \circ \gamma
	\end{equation}
	where $\Omega$ is a geometric constant and $R$ is the scalar curvature function of $M$.  
	
	We note here that \eqref{eqn:curveeqn} is actually equivalent to a system of two equations that we can obtain by expanding and projecting parallel and perpendicular to $\dot \gamma$.  That is,
\begin{equation}
	\label{eqn:curveeqntwo}
	\begin{aligned}
	f \nabla_{\dot \gamma} \dot \gamma &= \Omega r^2 \big( \nabla R \circ \gamma \big)^\perp \\
	\frac{\dif f }{\dif t} &= \Omega r^2 \langle \nabla R \circ \gamma , \dot \gamma \rangle  \, .
\end{aligned}
\end{equation}
Since $\langle \nabla R \circ \gamma , \dot \gamma \rangle  = \frac{\dif  }{\dif t} \big( R \circ \gamma \big)$ the second equation implies that there is a constant $c$ so that $f(t) = \Omega r^2 ( R\circ \gamma(t) + c )$.  Hence it is necessary only to solve the first of these two equations.
\end{subequations}
		
	\item The following boundary conditions hold for the various pairs $(\gamma, f)$.
	\begin{subequations}
	\label{eqn:curvebdconds}
	\begin{itemize}
		\item For each $p \in \mathcal V$ such that there is only one edge $\gamma \in \mathcal E$ emanating from $p$, we let $f$ be the function corresponding to $\gamma$ and assume without loss of generality that $\gamma$ is parametrized to begin at $p$.  Then
		\begin{equation}
			f(0) = 0 \qquad \mbox{as well as} \qquad \dot \gamma(0) \; \| \; \nabla R (p) \quad \mbox{and} \quad \langle \dot \gamma(0) , \nabla R(p) \rangle < 0 \, .
		\end{equation}
		
		\item For each remaining $p \in \mathcal V$, we let $\gamma_1, \ldots, \gamma_K \in \mathcal E$ be the $K \geq 2$ edges emanating from $p$ and we assume without loss of generality that these curves are parametrized to begin at $p$.  If $f_1, \ldots, f_K$ are the corresponding functions, then
		\begin{equation}
			\sum_{i=1}^K \Big(  f_i(0) + \Omega r^3 \langle \nabla R(p), \dot \gamma_i(0) \rangle \Big) \dot \gamma_i(0) = \Omega r^3 \nabla R (p) \, .
		\end{equation}
	\end{itemize}
	\end{subequations}
	
	\item Let $\Gamma_t$ be a continuous one-parameter family of graphs such that $\Gamma_t \rightarrow \Gamma$ where we mean the standard notion of convergence for one-dimensional varieties.  Then each curve $\gamma_t \in \Gamma_t$ with $\gamma_t \rightarrow \gamma \in \Gamma$ carries an infinitesimal deformation vector field $X_{\gamma}$.   We will also assume that no $X_{\gamma}$ belongs to the kernel of the linearization of the operator given in \eqref{eqn:curveeqn} corresponding to $\gamma$.

\end{enumerate}

The main theorem that will be proven in this paper is that a graph $\Gamma$ satisfying the conditions (1) -- (4) above is the condensation set of a sequence of CMC surfaces.

\begin{nonumthm}
	Let $\Gamma$ be a graph satisfying conditions (1) -- (4) above.  Then there is a sequence $r_j \rightarrow 0$ and a family of CMC surfaces $\Sigma_j$ where each $\Sigma_j$ has mean curvature $2/r_j$ and is contained in a tubular neighbourhood of radius $2 r_j$ of $\Gamma$.  
\end{nonumthm}

The proof of this result can be outlined as follows.  We use the standard machinery of gluing constructions, which constructs the CMC surfaces condensing onto $\Gamma$ in several steps.  First, we position spherical building blocks (well-chosen small perturbations of spheres) of radius $r$ end-to-end along each of the edges in $\Gamma$ and glue them to each other by means of optimally matched, re-scaled and truncated catenoids.   This yields an approximately CMC surface $\tilde \Sigma_r (\Gamma)$. We also construct neighbouring approximately CMC surfaces whose spherical constituents and necks are displaced by arbitrary small amounts.  Next, for each of these approximately CMC surfaces, we solve the constant mean curvature equation up to a remainder term in the approximate co-kernel of the linearized mean curvature operator of $\tilde \Sigma_r (\Gamma)$.  We then choose $\Gamma$ so that the highest-order term (in powers of $r$) of the remainder vanishes --- which is known as finding a \emph{balanced} initial approximately CMC surface.  Finally, we vary this $\Gamma$ slightly in order to find a surface for which the remainder term vanishes identically.

The key feature of the proof above is that Conditions (1) -- (4) given above translate into the equations that must be solved in order to find the graph $\Gamma$ that yields a balanced initial approximately CMC surface.   To explain this in more detail, we must first describe how the balancing equations are derived.  First let us consider a CMC surface $\Sigma$ with mean curvature $h$ in $(M, g)$ and suppose that $\mathcal U$ and $\mathcal W$ are open sets in $\Sigma$ and $M$, respectively, such that $\partial \bar{\mathcal W} = \bar{\mathcal U} \cup Q$ for some surface-with-boundary $Q$. The first variation formula for the area of $\mathcal U$ with the volume of $\mathcal W$ fixed relative to the one-parameter family of diffeomorphisms $\phi_t$ generated by a vector field $V$ on $M$ then implies 
\begin{align}
	\mylabel{eqn:balancing}
	0 &= \left. \frac{\dif}{\dif t} \Big( \mathit{Area} \big( \phi_t (\Sigma \cap \bar{ \mathcal U} ) \big) - h \mathit{Vol} \big( \phi_t (\bar{ \mathcal W}) \big) \Big) \right|_{t=0} -  \int_{\partial \mathcal U} g(\nu, V) + h \int_Q g(N, V) \, ,
\end{align}
where $\nu$ is the unit normal vector field of $\mathcal V$ in $\Sigma$ and $N$ is the unit normal vector field of $Q$ in $M$.  Next we consider the case when $\Sigma$ has approximately constant mean curvature, as is the case for our approximate solutions $\Sigma_r(\Gamma)$.  It turns out that the perturbation argument for finding a nearby exactly CMC surface works only when the right hand side of \eqref{eqn:balancing} is sufficiently close to zero for all vector field $V$ in the approximate co-kernel of the linearized mean curvature operator.  Since $\Sigma_r(\Gamma)$ for sufficiently small $r$ consists of spherical pieces joined by small necks, these vector fields are the approximate translation vector fields (defined as the coordinate vector fields of a  normal coordinate chart centered at a spherical piece) multiplied by cut off functions that cause them to vanish on all but one of the spherical pieces.  Evaluating the right hand side of \eqref{eqn:balancing} on such a vector field corresponding to the spherical piece centered at $p \in \Gamma$ yields
\begin{equation}
	\label{eqn:balancingtwo}
	\mathit{r.h.s.} = 
\sum_{\mbox{\scriptsize all necks}} r \eps_i \langle \eta_i , V \rangle - \Omega r^{4} \langle \nabla R (p), V \rangle + \mathcal O(r \eps^2) + \mathcal O(r^{6})
\end{equation}
where $\eta_i$ is the unit vector pointing along the geodesic from the center of the spherical piece in question to the $i^{\mathrm{th}}$ neck that connects this spherical piece to its neighbours, $r \eps_i$ is the width of this neck and $\eps := \max_i \{ \eps_i \}$.  Also, $\nabla R(p)$ is the gradient of the scalar curvature of $M$ at $p$, the brackets $\langle \cdot, \cdot \rangle$ denote the ambient metric at $p$ and $\Omega$ is a constant independent of $r$.  

In order to find $\Gamma$ so that the right hand side of \eqref{eqn:balancingtwo} vanishes identically, we proceed as follows.  Let $\overline{\mathcal V} $ denote the set of centers of all the spherical pieces used in the construction of $\Sigma_r(\Gamma)$.   We first find $\Gamma$ so that the leading order terms in \eqref{eqn:balancingtwo} all vanish, namely that the equations
\begin{equation}
	\label{eqn:balancingthree}
	0= \sum_{\mbox{\scriptsize all necks}}  \eps_i  \eta_i - \Omega r^{3} \nabla R (p)
\end{equation}
hold for all $ p \in \overline{\mathcal V}$.  We then perturb $\Gamma$ slightly so that the remaining small errors can be made to vanish as well.  This latter step requires a \emph{non-degeneracy condition}:  that the map which takes graphs near $\Gamma$ to values of the right hand side of equation \eqref{eqn:balancingtwo} be locally surjective.  

The requirements for the procedure above can be translated directly into Conditions (1) -- (4) of the Main Theorem.  To do so, consider the family of equations of the form \eqref{eqn:balancingthree} corresponding to the interior of a given curve $\gamma \subseteq \Gamma$.  Dividing through by $r$, this equation reads
\begin{equation}
	\label{eqn:balancingfour}
	0= \frac{1}{r} \big( \eps^+ \eta^+ + \eps^- \eta^- \big) - \Omega r^{2} \nabla R (p)
\end{equation}
where $\eta^\pm$ are the unit vectors of pointing in the direction of the neck ahead $(+)$ and behind $(-)$ the point $p$ on the curve $\gamma$ while $\eps^\pm$ are the corresponding neck sizes.   We can now view this equation as the discretization of a differential equation.  Up to additional small error terms, we have $\eta^\pm = \pm \dot \gamma(p \pm r)$ and if we introduce a neck-size function $f : \gamma \rightarrow \R$ then we have $\eps^\pm = f(p \pm r)$.  Hence the right hand side of \eqref{eqn:balancingfour} is equal to 
$$\nabla_{\dot \gamma} \big( f \dot \gamma \big) - \Omega r^2 \nabla R (p)$$
up to additional small error terms.  The vanishing to highest order in $r$ of \eqref{eqn:balancingfour} is thus equivalent to Condition (2) of the Main Theorem. Condition (1) is now the regularity condition which ensures that the various estimates of $\Sigma_r(\Gamma)$ that are needed in the perturbation theory remain uniform.  Condition (3) is obtained by expressing the equations that must hold at the boundaries of all the curves in $\Gamma$ as boundary conditions for the the pairs $(\gamma , f)$.  Finally, Condition (4) is the required non-degeneracy condition.

\paragraph{Comments on the existence conditions.}  The task of finding a non-trivial network of curves $\Gamma$ satisfying the existence conditions above is of course still open.   (The curve used in Butscher- Mazzeo \cite{memazzeo} should be seen as a trivial solution of these conditions because there the gradient of the scalar curvature points along the curve.) We will not attempt to solve in any general way the existence conditions in this paper.  Instead, we will point out a key feature of the equation \eqref{eqn:curveeqn} which should serve as a starting point for the investigation of the existence and properties of its solutions.

\begin{prop}
	Curves satisfying equation \eqref{eqn:curveeqn} are critical points of the functional
	$$\gamma \mapsto \int_\gamma \Big( \| \dot \gamma \|^2 + \Omega^2 r^4 ( R \circ \gamma + c )^2 \Big)$$
	for some constant $c$.
\end{prop}

\begin{proof}
We multiply both sides of \eqref{eqn:curveeqn} by $f$ and re-parametrize $\gamma$ so that $f \dot \gamma := \dot \sigma$ and hence $\| \dot \sigma \| = f =  \Omega r^2 ( R\circ \gamma + c )$.  This yields the equation 
$$\nabla_{\dot \sigma} \dot \sigma = \Omega r^2 \| \dot \sigma \| \nabla R \circ \gamma = \Omega^2 r^4 \big( R \circ \gamma + c \big) \nabla R \circ \gamma = \Omega^2 r^4 \nabla \big( R \circ \gamma + c \big)^2$$
which is precisely the Euler-Lagrange equation of the given functional.
\end{proof}

\paragraph{Acknowledgements.} I would like to thank Rafe Mazzeo, Frank Pacard,  Harold Rosenberg and Richard Schoen for interesting discussions during the course of this work.  I would also like to thank Harold Rosenberg for his generous hospitality during my stay at IMPA in Brazil in May 2009.

\section{Constructing a Family of Initial Surfaces}

Let $\Gamma$ be a graph as defined in Section \ref{sec:intro} with edges $\mathcal E := \{ \gamma_1, \ldots, \gamma_E\}$  and vertices $\mathcal V := \{ p_1, \ldots, p_N\}$. We now show how to construct a family of initial surfaces based on $\Gamma$ by positioning small slightly deformed geodesic spheres of radius $r$ end-to-end along the edges in $\mathcal E$, and gluing these to each other by inserting small catenoids.  Once a member of this family of surfaces is chosen, we also show how to construct a family of perturbations of this surface by allowing the locations of its spherical constituents to vary.  The construction procedure can be explained in several steps.

\paragraph{Step 1: The building blocks.}  We first construct building block surfaces in $\R^3$.  Identify $\Sph^2$ with the unit sphere centered at the origin in $\R^3$ and let $J_s := x^s \big|_{\Sph^2}$ be the restriction of the $s^{\mathit{th}}$ coordinate function to $\Sph^2$.  Set $\mathcal K_{\Sph^2} := \mathit{span} \{ J_1, J_2, J_3 \}$.  Choose a collection of points $q_1, \ldots, q_n  \in \Sph^2$ and small parameters $\eps_1, \ldots, \eps_n \in \R_+$ (which we'll abbreviate here by $\vec q$ and $\vec \eps$\,) and define the function $G_{\vec q, \vec \eps} : \Sph^2 \setminus \{ q_1, \ldots, q_n \} \rightarrow \R$ as the unique solution of the distributional equation 
\begin{equation}
	\label{eqn:greensdef}
	\Delta_{\Sph^2} G_{\vec q, \vec \eps} + 2 G_{\vec q, \vec \eps} = \sum_{i=1}^n \eps_i \delta_i + J \qquad \mbox{and} \qquad G_{\vec q, \vec \eps} \perp_{L^2} \mathcal K_{\Sph^2}
\end{equation}
where $\delta_i$ is the Dirac $\delta$-mass at $p_i$, while $\Delta_{\Sph^2}$ is the Laplacian of $\Sph^2$ with respect to the induced metric and $J \in \mathcal K$ is such that the right hand side of \eqref{eqn:greensdef} is $L^2$-orthogonal to $\mathcal K_{\Sph^2}$, thereby guaranteeing the existence of $G$.  Now let $\eps=\max\{ \eps_1, \ldots, \eps_n \}$ and let $r_\eps:= \eps^{3/4}$ (the reason for this choice will become clear in Step 2 below).  Define
$$\mathring S_r[ \vec q, \vec \eps] := \left\{ r (1 + G_{\vec q, \vec \eps}(\theta) ) \theta : \theta \in \Sph^2 \setminus \bigcup_{i=1}^n B_{r_\eps} (q_i) \right\}$$ 
which is the normal graph over $\Sph^2 \setminus \bigcup_{i=1}^n B_{r_\eps} (q_i) $ generated by the function $G_{\vec q, \vec \eps}$.  

We now transplant these building blocks to $M$.  Choose a point $p \in M$ and an orthonormal frame $E := \{ E_1, E_2, E_3 \}$ at $p$ from which we construct the inverse of the geodesic normal coordinate map $\phi_{p, E}^{-1} : \R^3 \rightarrow M$ via by $\phi_{p, E}^{-1} (x) := \exp_p (x^1 E_1 + \cdots  x^3 E_3)$ for $x = (x^1, x^2, x^3) \in B_R(0)$ where $R$ is some radius smaller than the injectivity radius of $M$ at $p$.  This allows us to obtain a building block in $M$ via the prescription
$$S_r[ p,  E, \vec q, \vec \eps] :=  \phi_{p, E}^{-1} (\mathring S_r[ \vec q, \vec \eps]) \, .$$
Finally, let us refer to the function $G_{\vec q, \vec \eps}$ as the \emph{generating function} of the building block $S_r[ p,  E, \vec q, \vec \eps]$.

\paragraph{Step 2: Gluing two building blocks together.}

Let us suppose that we are given one of the building blocks defined above, say $S := S_r[ p,  E, \vec q, \vec\eps]$ with $\vec q = (q_1, \ldots, q_n)$ and $\vec \eps = (\eps_1, \ldots, \eps_n)$.  We  now show how this building block can be glued to a second building block at the boundary component of $S$ corresponding to a chosen $q_i$.  First, view $q_i$ as a unit vector in $T_pM$ by means of the identification $q_i := \sum_{j=1}^3 q_i^j E_j$ and parametrize the geodesic emanating at $p$ in the direction $q_i$ by $\sigma_i (t) := \exp_p(t q_i)$.  Choose the point $p' := \sigma_i( (2+ \tau) r)$ where $\tau >0$ is a parameter satisfying $\tau = \mathcal O( \eps \log(1/\eps))$. Next, choose an orthonormal frame $E' := \{ E_1', E_2', E_3' \}$ at $p'$ such that $E_1' = - \dot \sigma_i((2+ \tau) r)$.  Finally, let $S' := S_r [p', E', \vec q^{\, \prime} ,\vec \eps^{\, \prime}]$ where $\vec q^{\, \prime} := (q_1', \ldots, q_{n'}')$ has $q_1' := (1, 0, 0)$ while $q_2', \ldots, q_{n'}'$ and $\vec \eps^{\, \prime} := ( \eps_1', \ldots, \eps_{n'}')$ are at this point arbitrary.

Let $p^\flat := \sigma_i ( ( 1 + \tau/2) r) $ denote the point half-way between $p$ and $p'$ and let $V^\flat := \dot \sigma_i (( 1 + \tau/2) r)$.  Choose an orthonormal frame at $p^\flat$ of the form $E^\flat := \{ V^\flat, E_2^\flat, E_3^\flat \}$ and consider the image of a neighbourhood of $p^\flat$ that contains both $S$ and $S'$ under the geodesic normal coordinate mapping $\phi_{p^\flat, E^\flat}$.  Re-scaling the image of this neighbourhood by a factor of $1/r$, we find that  $p^\flat$ is mapped to the origin, the geodesic segment $\sigma_i$ is mapped to the $x^1$-axis and the images of $S$ and $S'$, near the origin when $r$ is small, can be represented as two graphs over the $(x^2, x^3)$-plane of the form 
$$ \{ ( x^1, x^2, x^3) : x^1 = F_{\mathit{sph}}( \| (x^2, x^3)  \|)\} \qquad \mbox{and} \qquad  \{ ( x^1, x^2, x^3) : x^1 = F_{\mathit{sph}}' ( \|  (x^2, x^3) \|) \}$$
respectively.  One can check that the Taylor series expansions of the generating functions of $S$ and $S'$ near their singular points imply expansions for $F_{\mathit{sph}}$ and $F_{\mathit{sph}}'$ near zero of the form
\begin{align*}
	F_{\mathit{sph}}'(\bar x) &= \frac{\tau}{2 } + \eps_1'  \big( c' + C' \log (\bar x) \big) + \mathcal O(\bar x^2) + \mathcal O(|\eps'| \bar x^2)\\
	F_{\mathit{sph}}( \bar x) &= -\frac{\tau}{2 } - \eps_i \big( c + C \log ( \bar x ) \big) + \mathcal O( \bar x ^2) + \mathcal O(\eps \bar x^2) 
\end{align*}
for $\bar x > 0$.  Here, $c, c', C, C'$ are constants.
 
We will now show how to construct a six-parameter family of interpolating surfaces between the images of $S$ and $S'$ by gluing $S$ and $S'$ together using rotations, translations and re-scalings of a standard catenoid (i.e.\ the cylindrically symmetric, two-ended minimal surface in $\R^3$).  We first show how to construct the ``optimal'' member of this family.  To begin, consider a catenoid with its axis of symmetry along the $x^1$-axis, scaled by a factor of $\eps^\flat$, and translated by an amount $d^\flat$ along this axis, where $\eps^\flat$ and $d^\flat$ are yet to be determined.  The $x^1 > 0$ end of this catenoid is given by
\begin{align*}
	x^1 = F_{\mathit{neck}}^{ \eps^\flat, d^\flat, +}  (x^2, x^3)  :=&\; \eps^\flat \arccosh(  \| (x^2, x^3) \| / \eps^\flat) + \eps^\flat d^\flat \\
	=&\; \eps^\flat  \big( \log(2) - \log(\eps^\flat ) \big) + \eps^\flat \log( \| (x^2, x^3) \|) + \eps^\flat d^\flat + \mathcal O( (\eps^\flat)^3 / \| (x^2, x^3)\|^2)
\end{align*}
near the origin; and the $x^1 < 0$ end of this catenoid is given by
\begin{align*}
	x^1 = F_{\mathit{neck}}^ {\eps^\flat, d^\flat, -}( x^2, x^3)  :=&\; - \eps^\flat \arccosh(  \| (x^2, x^3) \| / \eps^\flat) + \eps^\flat d^\flat \\
	=&\; - \eps^\flat  \big( \log(2) - \log(\eps^\flat ) \big) - \eps^\flat \log( \| (x^2, x^3) \|) + \eps^\flat d^\flat + \mathcal O( (\eps^\flat)^3 / \| \bar x\|^2) 
	\end{align*}
near the origin.  Optimal matching of the asymptotic expansions of the ends of the catenoid with the asymptotic expansions of the images of $S$ and $S'$ then requires that the equations
\begin{equation}
	\label{eqn:gluingmatch}
	\eps_i = \frac{\eps^\flat}{C} \qquad \quad
	\eps_{1}' = \frac{\eps^\flat}{C'} \qquad \quad
	d^\flat = \frac{1}{2} \left( \frac{c'}{C'} - \frac{c}{C}\right) \qquad \mbox{and} \qquad  \tau = \Lambda(\eps^\flat) 
\end{equation}
hold, where
\begin{equation}
	\label{eqn:neckseprel}
	\Lambda (\eps^\flat) :=  \eps^\flat \left( 2 \big( \log(2) - \log(\eps^\flat) \big) -  \frac{c'}{C'} - \frac{c}{C} \right) \, . 	
\end{equation}
These equations imply that $\eps_i$ and $\tau$ together determine the parameters of the optimally matched neck $\eps^\flat$ and $d^\flat$ as well $\eps_1'$.  In other words, the building block $S$ and the spacing $\tau$ determines the parameters of the neighbouring building block $S'$ at the point of connection with $S$ as well as the neck interpolating between $S$ and $S'$.  Note also that the matching between the images of $S$ and $S'$ and the neck defined by this choice of parameters is \emph{most} optimal in the region of the $(x^2, x^3)$-plane where the error quantity $\mathcal O(\| (x^2, x^3) \|^2) + \mathcal O( (\eps^\flat)^3 / \| (x^2, x^3) \|^2)$ is smallest.  It is easy to check that this occurs when $\| (x^2, x^3) \| = \mathcal O (\max (\eps^{3/4}, (\eps')^{3/4} ) )$.  To complete the gluing process, we define a smooth, monotone cut-off function $\chi: [0,\infty) \rightarrow [0,1]$ which equals one in $[0,\frac{1}{2}]$ and vanishes outside $[0,1]$. Next, we define the functions $\tilde F^\pm : B_1(0) \rightarrow \R$ by
\begin{equation}
	\label{eqn:interpfn}
	\tilde F^\pm (x^2, x^3)  := \chi (\| (x^2, x^3) \| / (\eps^\flat)^{3/4}) F_{\mathit{neck}}^{\eps^\flat, d^\flat, \pm} (x^2, x^3 ) + \big( 1 - \chi(\| (x^2, x^3) \| / (\eps^\flat)^{3/4}) \big) F_{\mathit{sph}}^\ast (x^2, x^3)
\end{equation}
where $F_{\mathit{sph}}^\ast = F_{\mathit{sph}}$ if $\ast $ is $ -$ and $F_{\mathit{sph}}^\ast = F_{\mathit{sph}}' $ if $\ast $ is $+$.  Now define the interpolating surface as
\begin{align*}
	\tilde N &:= \phi_{p^\flat, V^\flat}^{-1} \Big( \big\{ ( \tilde F^+(x^2, x^3), x^2, x^3) : \eps^\flat \leq \| (x^2, x^3) \| \leq (\eps^\flat)^{3/4} \} \\
	&\qquad \qquad \qquad \qquad \cup  \{ ( \tilde F^-(x^2, x^3), x^2, x^3) : \eps^\flat \leq \| (x^2, x^3) \| \leq (\eps^\flat)^{3/4}  \big\} \Big) \, .
\end{align*}
The surface obtained by gluing $S$ to $S'$ can therefore now be defined as $S \cup \tilde N \cup S'$.  

Next, we explain how to construct interpolating surfaces close to the optimal one we have just defined, where the choice of catenoid neck can vary in a six-parameter family of choices.  Let $N$ denote the unit-scale catenoid with its axis of symmetry along the $x^1$ axis and centered on the origin.  Then the catenoid used in the optimal interpolation is thus $\eps^\flat \big(  N + (d^\flat, 0, 0) \big)$.  Let $R_{\theta_2, \theta_3}$ be the rotation by an angle $\theta_2$ about the $x^2$-axis followed by a rotation by angle $\theta_3$ about the $x^3$-axis.  Let $\eps \in \R$ be a small real number and let $\vec d := (d_1, d_2, d_3) \in \R^3$ be a vector of small norm.  Define the catenoid $N[d_1, d_2, d_3, \theta_2, \theta_3, \eps] := (1 + \eps) \eps^\flat R_{\theta_2, \theta_3} \big( N + (d^\flat, 0, 0) + \vec d \,\, \big)$. If $|\eps| + |\theta_2| + |\theta_3| + \| \vec d \|$ is sufficiently small, then the $x^1>0$ and $x^1<0$ ends of $N[d_1, d_2, d_3, \eps, \theta_2, \theta_3]$ can still be written as graphs over the $(x^2, x^3)$-plane.  We then define $\tilde N[d_1, d_2, d_3, \eps, \theta_2, \theta_3]$ as the interpolating surface obtained by replacing the functions $F_{\mathit{neck}}^{\eps^\flat, d^\flat, \pm}$ by the graphing functions of $N[d_1, d_2, d_3, \theta_2, \theta_3, \eps] $ in the definition of equation \eqref{eqn:interpfn}.  We get our six-parameter family of glued surfaces $S \cup  \tilde N[d_1, d_2, d_3, \eps, \theta_2, \theta_3] \cup S'$ in this way.

\paragraph{Step 3: Constructing the initial surface.} We first produce a family of piecewise-geodesic approximations of $\Gamma$ that are parametrized by a collection of small, positive separation parameters $\tau_e^{(s)}$ associated to each $\gamma_e \in \mathcal E$.   We can do this inductively as follows.  Let $\gamma_e \in \mathcal E$ and $\tau_e^{(1)}, \ldots, \tau_e^{(S_e)}$ be given and let $q_e^{(0)} = \gamma_e(0)$ be the vertex at the beginning of $\gamma_e$.  Now determine $q_e^{(1)}$ as the unique point in $\gamma_e$ located a distance $(2 + \tau_e^{(1)}) r$ from $q_e^{(0)}$.  Then determine $q_e^{(2)}$ as the unique point in $\gamma_e$ located a distance $(2 + \tau_e^{(2)}) r$ from $q_e^{(1)}$.  Proceed in this way until we reach the vertex at the end of $\gamma_e$, namely $q_e^{(S_e)} = \gamma_e (|\gamma_e|)$ where $|\gamma_e|$ is the length of $\gamma_e$.   Note that the last separation parameter $\tau_e^{(S_e)}$ is a function of $|\gamma_e|$ and $r$ and $\tau_i^{(s)}$ for $s = 1, \ldots, S_e - 1$ and furthermore, if $\tau_e^{(S_e)}$ is to be of size $o(r)$ then we must choose $r$ belonging to a sequence of small intervals accumulating at zero and a corresponding diverging sequence of integers $S_e = S_e(r)$.

We can complete the construction of the initial surface based on the piecewise-geodesic approximation found above.  We need only position the appropriate building block at each $q_e^{(s)}$ and connect them with the appropriate necks in order to end up with a smooth surface that we will call $\tilde \Sigma_r(\tau)$ with $\tau$ indicating the dependence on the separation parameters chosen above.  Indeed, if $q_e^{(s)}$ is attached to neighbours $q_i$ by geodesic segments of length $(2 + \tau_i)r$ for $i = 1, \ldots, n$,  then we use the building block $S_r(q_e^{(s)}, q_1, \ldots, q_n, \Lambda_1(\tau_1) / C_1, \ldots, \Lambda_n (\tau_n)/C_n )$ where $\Lambda_i$ and $C_i$ are the quantities appearing in \eqref{eqn:gluingmatch} corresponding to $q_i$. The separation parameters also determine the optimal interpolating catenoids as we have seen.

\paragraph{Step 4: Constructing perturbations of the initial surface.} Let $\overline{\mathcal V} := \{ q_1, \ldots, q_N \}$ be the union of all points in all edges of $\Gamma$ where we have placed building blocks during the construction of $\tilde \Sigma_r(\tau)$ in Steps (1) -- (3) above.  Note that $\mathcal V \subset \overline{\mathcal V}$. For each $q \in \overline{\mathcal V}$ introduce a tangent vector $W_q \in T_q M$  of small length and denote by $W  \in \prod_{q \in \overline{\mathcal V}} T_q M$ the quantity $W := (W_{q_1}, \ldots )$.  Form a new piecewise-geodesic approximation of $\Gamma$ as follows.  First move each $q \in \overline{\mathcal V}$ to $\exp_q(r W_q)$.  Then for each pair $q$ and $ q' \in \overline{\mathcal V}$ that are connected by a geodesic segment in the piecewise-geodesic approximation of $\Gamma$ constructed in Step 3, we connect the corresponding points $\exp_q(r W_q)$ and $ \exp_{q'}(r f_{q'})$ by a geodesic segment.  Now we can proceed as in Step (3) and place building blocks at each point $\exp_q(r W_q)$ for $q \in \overline{\mathcal V}$, where the various parameters of this building block are determined by the new direction vectors and the lengths of the geodesic segments attached to $\exp_q(r W_q)$.  For each $q^\flat \in \overline{\mathcal{V}}^{\, \flat}$, introduce $\Xi_{q^\flat} \in \R^6$ of small length and denote by $\Xi\in \prod_{q^\flat \in \overline{\mathcal V}^{\, \flat}} \R^6$ the quantity $\Xi := (\Xi_{q_1^\flat}, \ldots ) $.  Now glue the building blocks adjoining the edge containing $q^\flat$  together using the interpolating catenoid whose parameters are given by the components of $\Xi_{q^\flat}$.  Once we've completed the process that we have just described, we end up with a smooth surface that is a small perturbation of the surface constructed in Step 3.  Denote this surface by  $\fullapsol$.

\section{Deforming an Initial Surface into an Almost-CMC Surface} 

Let $\apsol:= \fullapsol$ be a fixed, initial surface constructed as above.  In this section of the paper, we show how to deform  $\apsol$ into an \emph{almost-CMC} surface by solving a partial differential equation using a contraction mapping argument in a space of $C^{2, \alpha}$ functions on $\apsol$.  This means that we will be able to find a small normal deformation of $\apsol$ that is determined by the solution of the PDE which yields a surface whose mean curvature is the constant $2/r$ plus a small, well-controlled error term.  The arguments required here are all fairly standard; in particular, they are carried out in full detail in \cite{memazzeo} in the special case where $\Gamma$ consists of one curve and the ambient manifold $M$ possesses a high degree of symmetry.  Hence the arguments will be stated here in an abbreviated sense for the convenience of the reader.

\subsection{Function Spaces and Norms} 

\paragraph{Partitions of unity.}  We define once and for all a family of partitions of unity for $\apsol$ that will be used throughout the paper.  In what follows, let $\overline{\mathcal V}$ be the set of end-points of all the geodesics segments used in the construction of $\apsol$ as above, and also let $\overline{\mathcal V}^{\, \flat}$ be the set of mid-points of all these geodesic segments --- which is where the various necks have been placed.  Let $\eps_{q^\flat}$ be the scale parameter used in the definition of the neck placed at $q^\flat \in \overline{\mathcal V}^{\, \flat}$ and put $\eps = \max \{\eps_{q^\flat} : q^\flat \in \overline{\mathcal V}^{\, \flat} \}$.  We begin by defining subsets of $\apsol$.  First, for any $p \in M$ and $0< \sigma < \sigma_0$ for some fixed, small threshold $\sigma_0$, we define $\mathit{Ann}_\sigma (p) := B_\sigma(p) \setminus \overline{B_{\sigma/2}(p)}$.  Now let $q^\flat \in \overline{\mathcal V}^{\, \flat}$ and define the subsets:
\begin{align*}
	\mathcal N_{q^\flat}^{\, \sigma} := \apsol \cap B_\sigma (q^\flat) \qquad \mbox{and} \qquad \mathcal T_{q^\flat}^{\, \sigma} &:= \apsol \cap \mathit{Ann}_\sigma(q^\flat) 
\end{align*}
and write $\mathcal T_{q^\flat}^{\, \sigma} := \mathcal T_{q^\flat}^{\, \sigma, +} \cup \mathcal T_{q^\flat}^{\, \sigma, -}$ where $\mathcal T_{q^\flat}^{\, \sigma, \pm}$ are the two disjoint components of $\mathcal T_{q^\flat}^{\, \sigma}$ `ahead' and `behind' the point $q^\flat$ with respect to the parametrization of the edge to which $q^\flat$ belongs.  If $\sigma_0$ is sufficiently small, then the set $\apsol \setminus \bigcup_{q^\flat \in \overline{\mathcal V}^{\, \flat}} \! \Big( \overline{ \mathcal N_{q^\flat}^{\, \sigma} \cup \mathcal T_{q^\flat}^{\, \sigma} } \Big)$ is the disjoint union of open sets corresponding to each $q \in \overline{\mathcal V}$.  In this way we can unambiguously define $\mathcal S_{q}^{\, \sigma}$ via the prescription
$$ \bigcup_{q \in \overline{\mathcal V}} \mathcal S_{q}^{\, \sigma} := \apsol \setminus \bigcup_{q^\flat \in \overline{\mathcal V}^{\, \flat}} \! \Big( \overline{ \mathcal N_{q^\flat}^{\, \sigma} \cup \mathcal T_{q^\flat}^{\, \sigma} } \Big)  \, . $$

We now define the smooth, positive cut-off functions making up the partition of unity subordinate to the cover of $\Sigma_r(W)$ given by the sets $\mathcal N_{q^\flat}^{\, \sigma}$, $\mathcal T_{q^\flat}^{\, \sigma}$ and $\mathcal S_{q}^{\, \sigma}$ for one fixed $\sigma < \sigma_0$.  Given $q^\flat \in \overline{\mathcal V}^{\, \flat}$ and $q \in \overline{\mathcal V}$ connected to neck regions centered at $q_i^\flat$ for $i = 1, \ldots, n$, we define 
\begin{align*}
	\chi^{\sigma}_{\mathit{neck}, q^\flat} (x) &:= 
	\begin{cases}
		1 &\quad \! x \in \mathcal N^{\, \sigma}_{q^\flat}  \\
		\mbox{\small $\mathit{Interpolation}$} &\quad\! x \in  \mathcal T_{q^\flat}^{\, \sigma} \\
		0 &\quad\! \mbox{elsewhere}
	\end{cases} \\[1ex]
	\chi^{\sigma}_{\mathit{sph}, q} (x) &:= 
	\begin{cases}
		1 &\quad\! x \in \mathcal S^{\, \sigma}_q \\
		\mbox{\small $\mathit{Interpolation}$} &\quad\! x \in \bigcup_{i=1}^n \mathcal T_{q_i^\flat}^{\, \sigma, -} \\
		0 &\quad\! \mbox{elsewhere}	
	\end{cases}	
\end{align*}
with the property that $\sum_{q \in \overline{\mathcal V}} \chi^{\sigma}_{\mathit{sph}, q} + \sum_{q^\flat \in \overline{\mathcal V}^{\, \flat}} \chi^{\sigma}_{\mathit{neck}, q^\flat} = 1$ for all $\sigma$.

\paragraph{Weighted norms.}  We introduce the standard weighted norm that achieves control of functions in the neck regions and asymptotic ends of $\apsol$. We first define a weight function $\zeta_{r} : \apsol \rightarrow \R$ as follows.  For each $q^\flat \in \overline{\mathcal V}^{\, \flat}$, choose $\sigma(q^\flat) := \mathcal O( r \eps_{q^\flat} \log(1/|\eps_{q^\flat}|) )$ on the order of the width of the neck regions at $q^\flat$.  Let $\delta_{\mathit{neck}, q^\flat} :  \mathcal N_{q^\flat}^{\, \sigma(q^\flat)} \rightarrow \R$ be such that $\delta_{\mathit{neck}, q^\flat} (x)$ gives the distance of a point $x \in \mathcal N_{q^\flat}^{\,  \sigma(q^\flat)}$ to the narrowest part of $\mathcal N_{q^\flat}^{\,  \sigma(q^\flat)}$.    We now define the weight function by
\begin{equation*}
	\zeta_{r}(p) := 
	\begin{cases}
		\sqrt{ \eps_{q^\flat}^2 +  [\delta_{\mathit{neck}, q^\flat}(x) ]^2}  &\quad x \in \mathcal N_{q^\flat}^{\, \sigma(q^\flat)} \\
		\mbox{\small $\mathit{Interpolation}$} &\quad x \in \mathcal T_{q^\flat}^{\, \sigma(q^\flat)} \\
		r &\quad \mbox{elsewhere} 
	\end{cases}
\end{equation*}
where the interpolation is such that $\zeta_r$ is smooth and monotone in the region of interpolation, and has appropriately bounded derivatives. 

Next we define the norms we will use.  First, if $\mathcal U$ is any open subset of $\apsol$ and $A$ is any tensor field on $\mathcal U$, define 
\begin{equation*}
	| A |_{0, \mathcal U} := \sup_{x \in \,  \mathcal U} \Vert A(x) \Vert \\
	\qquad \mbox{and} \qquad [A]_{\alpha, \, \mathcal U} := \sup_{x,x' \in \, \mathcal U} \frac{\Vert A(x') - \Xi_{x,x'} (A(x)) \Vert}{\mathrm{dist}(x,x')^\alpha} \, ,
\end{equation*}
where the norms and the distance function that appear are taken with respect to the induced metric of $\apsol$, while $\Xi_{x,x'}$ is the parallel transport operator from $x$ to $x'$ with respect to this metric.  Here we use the convention that $| \cdot |_{0, \mathcal U} = [ \cdot ]_{\alpha, \mathcal U} = 0 $ if $\mathcal U = \emptyset$.  For any function $f: \mathcal U \rightarrow \R$ we now define its $\ckan$ norm by
\begin{align*}
	|f|_{\ckan(\mathcal U)} &:= \sum_{i=0}^{k}  | \zeta_r^{i-\nu} \nabla^i f |_{0, \, \mathcal U } + [ \zeta_r^{k + \alpha -\nu} \nabla^k f ]_{0, \, \mathcal U }  \, .
\end{align*}

\paragraph{Function spaces.} The function spaces that will be used in the remainder of the paper are simply the usual spaces $C^{k, \alpha}(\apsol)$, but endowed with the $\ckan$ norm defined above. This space will be denoted $\ckan (\apsol)$.  The norm will often be abbreviated by $| \cdot |_{\ckan}$ when there is no cause for confusion.

\subsection{Setting Up a Contraction Mapping Problem}
\label{sec:strategy}

Let $\mu : \ctan (\apsol) \rightarrow \mathit{Emb}( \apsol, M)$ be the exponential map of $\apsol$ in the direction of the unit normal vector field of $\apsol$ with respect to the backgroung metric $g$.  Hence $ \mu_{rf} \big(\apsol \big)$  is the scaled normal deformation of $\apsol$ generated by $f \in \ctan (\apsol)$.  The equation
\begin{equation}
	\label{eqn:tosolve}
	H \big[ \mu_{rf} \big(\apsol \big) \big]  = \frac{2}{r}
\end{equation}
selects $f \in \ctan(\apsol)$ so that $\mu_{rf}( \apsol)$ has constant mean curvature equal to $\frac{2}{r}$.    At this stage, it will not be possible to solve \eqref{eqn:tosolve} exactly, but using a contraction mapping argument together with the choice $\nu \in (1, 2)$ the equation \eqref{eqn:tosolve} can be solved up to a \emph{co-kernel error term}.  This means that a solution of
$$H \big[ \mu_{rf} \big(\apsol \big) \big]  = \frac{2}{r} + \mathcal E$$
can be found, where $\mathcal E$ belongs to a subspace of functions associated to the approximate co-kernel of the linearized mean curvature operator of $\apsol$ that will be denoted $\tilde{\mathcal K}$ and specified below.  At first glance, the error $\mathcal E$ will come from terms in the solution procedure that are not sufficiently small.  However, the true reason for the presence of $\mathcal E$ is geometric and will be explained in Section \ref{sec:balancing}, where we show how to eliminate it by a suitable choice of $W$. 

We now outline how this will be done.  First we decompose the mean curvature operator of the deformed surface as 
$$H [  \mu_{rf} \big(\apsol \big)] = H [  \apsol ] + \mathcal L (rf) + \mathcal Q(rf) + \mathcal H(rf)$$
where $\mathcal L := \Delta + \| B \|^2$ is the dominant part of the linearized mean curvature operator, $\mathcal Q$ is the quadratic remainder part of the mean curvature and $\mathcal H (rf) := \mathit{Ric} \big|_{\mu_{rf} (\apsol)}  (N_f, N_f)$ is an additional small error term.  Then we construct a bounded parametrix  $\mathcal R : \coan (\apsol) \rightarrow \ctan (\apsol)$ satisfying $ \mathcal L \circ \mathcal R = \mathit{Id} + \mathcal E$ where $\mathcal E$ maps into $\tilde{\mathcal K}$.  Now the \emph{Ansatz} 
$$f :=  \frac{1}{r} \mathcal R \left( u - H \big[\apsol \big] + \frac{2}{r}  \right)$$
transforms the equation \eqref{eqn:tosolve} into the fixed-point problem
\begin{align}
	\label{eqn:cmapf}
	u & = -  {\mathcal Q} \circ  \mathcal R \left(u -  H \big[\apsol \big]  + \frac{2}{r} \right)  -  {\mathcal H} \circ  \mathcal R \left(u -  H \big[\apsol \big]  + \frac{2}{r} \right) \, .
\end{align}
up to the co-kernel error term in $\tilde{\mathcal K}$.  The remaining task is to show that the mapping 
$$\mathcal N_r : \coan (\apsol) \rightarrow \coan (\apsol)$$
given by the right hand side of \eqref{eqn:cmapf} is a contraction mapping onto a neighbourhood of zero containing $H \big[\apsol \big]  - \frac{2}{r}$.  Finally, we will show that a well-behaved solution of equation \eqref{eqn:tosolve} up to an error term can be found for any choice of $(\tau, W, \Xi)$ where $\Xi$ is sufficiently small, $W$ is contained in a compact subset of $ \prod_{q \in \overline{\mathcal V}} T_q M \setminus \boldsymbol{\Delta}$ where $\boldsymbol{\Delta}$ is the subset of $\prod_{q \in \overline{\mathcal V}} T_q M$ consisting of those $W$ for which the unit vectors associated to any pair of neighbouring vertices point in opposite directions (after parallel transport along the geodesic segment connecting the vertices in question so that the vectors can be compared), and $\tau$ is sufficiently small but all its components are positive and bounded away from zero by a fixed amount depending only on $r$.

\subsection{Structure of the Mean Curvature Operator}

In this section, we collect all the various structural facts about the mean curvature operator that will be used in subsequent sections of this paper.  We can more or less quote exactly the analogous results from \cite{memazzeo} since the setting here is in most respects identical to the setting of \cite{memazzeo}.  Nevertheless, a brief development of the results will be presented here in order to give the reader a sense of the structure of the partial differential equation that must be solved in order to deform $\apsol$ into a CMC surface.

\paragraph{Normal coordinate charts.}  Consider a geodesic normal coordinate chart centered at a point $p \in M$.  We can express the background metric in this chart as
$$ 
g := \mathring g + P :=  \big( \delta_{ij} + P_{ij}(x) \big) \, \dif x^i \otimes \dif x^j 
$$
where $\mathring g$ is the Euclidean metric and $P$ is the perturbation term.  It is well known that
\begin{equation}
	\label{eqn:metricexp}
	P_{ij}(x) = \frac{1}{3} \sum_{l,m} R_{iljm}(p) x^l x^m + 
\frac{1}{6} \sum_{l,m,n} R_{iljm;n}(p) x^l x^m x^n  + \mathcal O (\| x \|^4) \, .
\end{equation}
where $R_{ijkl} := \mathrm{Rm}( \frac{\partial}{\partial x^i}, \frac{\partial}{\partial x^j}, \frac{\partial}{\partial x^k}, \frac{\partial}{\partial x^l})$ and $R_{ijkl;m} := \bar \nabla_{\! \! \frac{\partial}{\partial x^m} } \mathrm{Rm}( \frac{\partial}{\partial x^i}, \frac{\partial}{\partial x^j}, \frac{\partial}{\partial x^k}, \frac{\partial}{\partial x^l})$ are components of the ambient Riemann curvature tensor $\mathrm{Rm}$ and its ambient covariant derivative $\bar \nabla \mathrm{Rm}$.  

\paragraph{Surfaces in normal coordinate charts.}  Now let $\Sigma$ be a surface (perhaps with boundary) contained in this chart.  We can now use the above expansion to derive corresponding expansions in terms of the background curvature at $p$ for any geometric object defined on $\Sigma$.  Here and in the rest of the paper, let $h, \Gamma, \nabla, \Delta, N, B, H$ be the induced metric, Christoffel symbols, covariant derivative, Laplacian, unit normal vector, second fundamental form, and mean curvature of $\Sigma$ with respect to the metric $g$, and let $\mathring h, \mathring \Gamma, \mathring \nabla, \mathring \Delta, \mathring N, \mathring B, \mathring H$ be these same objects with respect to the Euclidean metric. Near a point $x \in \Sigma$, let $\{ E_1, E_2\}$ be a local
frame for $T\Sigma$ induced by some coordinate system and denote by $Y := \sum_j x^j \frac{\partial}{\partial x^j}$ 
the position vector.  Define
\begin{align*}
	\mathcal P_{00} &:= P(\mathring N, \mathring N) \\
	\mathcal P_{0j} &:= P(\mathring N, E_j) \\
	\mathcal P_{ij} &:= P(E_i, E_j) \\
	\mathcal P_{ijt} &:= \tfrac{1}{2} \big( (E_i P) (E_j, E_t) + (E_j P)(E_i, E_t) - (E_t P) (E_i, E_j) \big) \\
	\mathcal P_{ij0} &:= \tfrac{1}{2} \big( (E_i P) (E_j, \mathring N) + (E_j P)(E_i, \mathring N) - (\mathring N P) (E_i, E_j) \big) \, .
\end{align*}
Straightforward geometric calculations now give us the following results.  

\begin{lemma}
	\label{lemma:basicexp}
	Let $\Sigma$ be a surface belonging to a geodesic normal coordinate chart of $M$ where the expansion \eqref{eqn:metricexp} is valid.  Then the induced metric of $\Sigma$ and the associated Christoffel symbols satisfy
	$$h_{ij} = \mathring h_{ij} + \mathcal P_{ij} \quad \mbox{ and } \quad h^{ij} =  \mathring h^{ij} - \mathring h^{is} \mathring h^{jt} \mathcal P_{st} + \mathcal O(\| Y \|^4) \quad \mbox{ and } \quad \Gamma_{ijk} = \mathring \Gamma_{ijk} + \mathcal P_{ijk} + \mathcal P_{0k} \mathring B_{ij} \, ,$$
	the normal vector of $\Sigma$ satisfies
	\begin{align*}
		N = \frac{ \mathring N - h^{ij} \mathcal P_{0j} E_i }{ \big( 1 + \mathcal P_{00} - \mathring h^{ij} \mathcal P_{0i} \mathcal P_{0j} \big)^{1/2} } = \big( 1 - \tfrac{1}{2} \mathcal P_{00} \big) \mathring N - \mathring h^{ij} \mathcal P_{0j} E_i + \mathcal O(\| Y \|^4)  \, ,
	\end{align*}
	the second fundamental form of $\Sigma$ satisfies
	\begin{align*}
		B_{ij} &= \big( 1 + \mathcal P_{00} - \mathring h^{ij} \mathcal P_{0i} \mathcal P_{0j} \big)^{1/2}  \mathring B_{ij} + \frac{\mathcal P_{ij0} - \mathring h^{kl} \mathcal P_{0k} \mathcal P_{ijl} }{  \big( 1 + \mathcal P_{00} - \mathring h^{ij} \mathcal P_{0i} \mathcal P_{0j} \big)^{1/2} } \\
		&= \big( 1 + \tfrac{1}{2} \mathcal P_{00} \big)\mathring B_{ij} + \mathcal P_{ij0}  + \mathcal B_{ij} (Y, \mathring B, \mathring N, E_1, E_2)  \, ,
	\end{align*}
	and finally the mean curvature of $\Sigma$ satisfies
	\begin{align*}
		H &=\big( 1 + \tfrac{1}{2} \mathcal P_{00} \big)  \mathring H + \mathring h^{ij} \mathcal P_{ij0}  - \mathring B^{ij} \mathcal P_{ij} + \mathcal H(Y, \mathring B, \mathring N, E_1, E_2)  \, .
	\end{align*} 
In the last two sets of identities,  $\mathcal B_{ij}$ and $\mathcal H$ are functions with
$$\max_{i,j} | \mathcal  B_{ij} (Y, \mathring B, \mathring N, E_1, E_2) | + | \mathcal H (Y, \mathring B, \mathring N, E_1, E_2) | \leq  C \| Y \|^3 (1 + \| Y \| \| \mathring B \| )$$
for a constant $C$ depending only on the curvature tensor of the ambient manifold.
\end{lemma}

\paragraph{Normal graphs in normal coordinate charts.}  Let us next suppose that the surface $\Sigma$ in the geodesic normal coordinate chart where the expansion \eqref{eqn:metricexp} is valid is deformed in the following way.  Choose a function $f : \Sigma \rightarrow \R$ and define $\mu_{f} : \Sigma \rightarrow \R^{3}$ to be the normal deformation of $\Sigma$ by $f$ with respect to the normal vector field of $\Sigma$ calculated using the background metric $g$.  Put $\Sigma_{f} := \mu_{f} (\Sigma)$. We would like to obtain a workable expression for the mean curvature of $\Sigma_f$ and understand its dependence on the function $f$ and the expansion of the metric $g$.  

However, straightforward comparison is not possible because in order to express $\Sigma_f$ as a normal deformation of $\Sigma$ with respect to the Euclidean metric, we can not use a different normal graphing function.  Let $\mathring \mu_{f} : \Sigma \rightarrow \R^{3}$ to be the normal deformation of $\Sigma$ by $f$ with respect to the normal vector field of $\Sigma$ calculated using the Euclidean metric.  Put $\mathring \Sigma_{f} := \mu_{f} (\Sigma)$.  The next lemma shows that we can replace $\Sigma_f$ by $\mathring \mu_{f_0} (\Sigma)$ for a new function $f_0 : \Sigma \rightarrow \R$ and then relate $|f|_{C^{2, \alpha}}$ to $|f_0|_{C^{2, \alpha}}$. 

\begin{lemma}
	\label{lemma:normalgraph}
	Let $\Sigma$ be a surface belonging to a geodesic normal coordinate chart of $M$ where the expansion \eqref{eqn:metricexp} is valid and let $\Sigma_{f} := \mu_{f} (\Sigma)$.  Under the assumption that $\mathit{diam}(\Sigma)$ is sufficiently small and $|f| \| \mathring B \| \ll 1$ on $\Sigma$, then there exists a function $f_0 : \Sigma \rightarrow \R$ so that $\Sigma_f = \mathring \mu_{f_0} (\Sigma)$.  Moreover, $f_0$ satisfies the estimate
	$$|f_0|_{C^{2, \alpha}(B_\theta)} \leq C |f|_{C^{2, \alpha}(B_{2\theta})} $$
where $C$ is a constant that depends only on the geometry of $M$ in the coordinate chart, while $B_\theta$ and $b_{2 \theta}$ is any pair of concentric balls of radii $\theta$ and $2 \theta$ contained in the coordinate chart.
\end{lemma}

We can now analyze the mean curvature operator and the second fundamental form of a surface of the form $\mathring \Sigma_{f} := \mathring \mu_{f} (\Sigma)$ in a geodesic normal coordinate chart, where $f: \Sigma \rightarrow \R$.  We start with the fact that the second fundamental form and mean curvature of $\mathring \Sigma_{f}$ with respect to the Euclidean metric can be decomposed as
\begin{equation}
	\label{eqn:eucquad}
	\begin{aligned}
		\mathring B [\mathring \Sigma_f ] &= \mathring B + \mathring B^{(1)} (f) +  \mathring B^{(2)} (f) \\
		\mathring H [\mathring \Sigma_f ] &= \mathring H  + \mathring{\mathcal L} (f) + \mathring{\mathcal Q} (f) 
	\end{aligned}
\end{equation}
where $\mathring{\mathcal L} := \mathring \Delta + \| \mathring B \|^2 $ and $\mathring B^{(1)}$ are linear operators while $\mathring Q$ and $\mathring B^{(2)}$ are second-order differential operators that are quadratic and higher in their arguments.  Expressions for these operators are well known and are given in \cite{memazzeo}.  Using Lemma \ref{lemma:basicexp} we can now derive from \eqref{eqn:eucquad} the analogous decomposition of the mean curvature of $\mathring \Sigma_{f}$ with respect to the background metric $g$ as well as obtain estimates for $f$.  This is carried out in \cite{memazzeo} and the result is the following.  We note that the requirement $| f| \| \mathring B \| + \| \mathring \nabla f \| \ll 1$ appearing below is actually very natural and we will show later that it holds for the functions we will be considering.

\begin{lemma}
	\label{lemma:prepestimates}
	Let $\Sigma$ be a surface belonging to a geodesic normal coordinate chart of $M$ where the expansion \eqref{eqn:metricexp} is valid and let $\mathring \Sigma_f := \mathring \mu_f(\Sigma)$.  Then the mean curvature of $\mathring \Sigma_f$ with respect to the background metric $g$ can be expanded as
	\begin{equation*}
	\label{eqn:meancurvfnexp}
	H[\mathring \Sigma_f] = H + \mathcal L(f) + \mathcal Q(f) + \mathcal H(f)
\end{equation*}
	where $\mathcal L$ is a linear operator, $\mathcal Q$ is a quadratic remainder and $\mathcal H$ is an error term.  Furthermore, the following estimates are valid.  First, $\mathcal L$ satisfies
\begin{align*}
	| \mathcal L(u) - \mathring{\mathcal L}(u) | &\leq C (1 + \| Y \| \| \mathring B \|) (|u| + \|Y \|  \| \mathring \nabla u \| + \| Y \|^2 \| \mathring \nabla^2 u \|) \, .
\end{align*}  
Under the assumption that $| f_i| \| \mathring B \| + \| \mathring \nabla f_i \| \ll 1$ for $i = 1, 2$, then $\mathcal Q$ satisfies
\begin{align*}
		 | \mathcal Q(f_1) -\mathcal Q(f_2) | & \leq C | f_1 - f_2 | \cdot \max_i \big(|f_i| \| \mathring B \|^3 + \|\mathring \nabla f_i \| \| \mathring B \|^2 +  \|\mathring \nabla f_i \| \| \mathring \nabla \mathring B \|  + \| \mathring \nabla^2 f_i \| \| \mathring B \| \big) \\
		&\qquad + C \| \mathring \nabla f_1 - \mathring \nabla f_2 \| \cdot \max_i \big(  | f_i | \| \mathring B \|^2 + \|\mathring \nabla f_i \| \| \mathring B \| + | f_i | \| \mathring \nabla \mathring B \| + \| \mathring \nabla^2 f_i \|  \big) \\
		&\qquad + C \| \mathring \nabla^2 f_1 - \mathring \nabla^2 f_2 \| \cdot \max_i \big( |f_i| \| \mathring B \| + \| \mathring \nabla f_i \| \big) \\
		&\qquad + C \| \mathring \nabla^2 f_1 - \mathring \nabla^2 f_2 \| \cdot \max_i \| \mathring \nabla f_i \|^2 + C \| \mathring \nabla f_1 - \mathring \nabla f_2 \| \cdot \max_i \| \mathring \nabla f_1 \| \| \mathring \nabla^2 f_i \| \\
		&\qquad +C | f_1 - f_2| \cdot \max_i \| \mathring \nabla f_i \| + C \| \mathring \nabla f_1 - \mathring \nabla f_2 \| \cdot \max_i | f_i | \\
		&\qquad + C \big(  | f_1 - f_2 |+ \| Y \| \| \mathring \nabla f_1 - \mathring \nabla f_2 \| \big) \cdot \max_{i} \big( | f_i |+ \| Y \| \| \mathring \nabla f_i \|\big) \, .
\end{align*}
Finally, under the assumption that $| f| \| \mathring B \| + \| \mathring \nabla f \| \ll 1$ then $\mathcal H$ satisfies
	$$| \mathcal H(f_1) -   \mathcal H(f_2)  | \leq C  \| Y \|^2  (1 + \| Y \| \| \mathring B \|)  \big(|f_1 - f_2| +  \|Y \| \| \mathring \nabla f_1 - \mathring \nabla f_2 \| + \|Y \|^2 \| \mathring \nabla^2 f_1 - \mathring \nabla^2 f_2 \| \big) \, .$$
In all of the estimates above, $C$ is a constant depending only on the curvature tensor of the ambient manifold at the center of the normal coordinate chart under consideration.
\end{lemma}

\subsection{Estimates of the Mean Curvature of the Approximate Solutions}
\label{sec:meancurvest}

The first step in the proof that the equation $H \big[\apsol \big] = \frac{2}{r}$ can be solved up to an error term is to estimate $\left| H [\apsol ] - \frac{2}{r} \right|_{\coan} $.  To do so, we need only trivially modify the estimates found in \cite{memazzeo} of this quantity for use in the present setting.  Therefore we give only a very abbreviated outline of the proof here.  In the what follows, set $ \eps := \max \{ \eps_q : q \in \overline{\mathcal V} \} \cup \{ \eps_{q^\flat} : q^\flat \in \overline{\mathcal V}^{\, \flat} \}$ and recall that $r_{\eps} = \mathcal O(r \eps^{3/4})$.  

\begin{prop}
	\label{prop:meancurvest}
	Suppose $\nu \in (1,2)$.  The mean curvature of $\apsol$ satisfies the estimate
	$$\left|  H [\apsol ] - \frac{2}{r} \right|_{\coan} \leq C \max \big\{ r^{3 - \nu}  , \, r^{1-\nu} \eps^{3/2 - 3\nu/4} \big\} $$
	for some constant $C$ independent of $r$ and $\eps$. Moreover, this estimate is uniform in $(\tau, W, \Xi)$ provided these satisfy the requirements set out in Section \ref{sec:strategy}.
\end{prop}

\begin{proof} The proof follows three steps.  The first step is to derive pointwise estimates for the mean curvature and second fundamental form of $\apsol$ with respect to the Euclidean background metric in the spherical and neck regions in the geodesic normal coordinates in which they are defined.  For this purpose we use the well-known explicit expressions of these quantities that are available.  The second step is to convert these estimates into pointwise estimates for the mean curvature with respect to the background metric $g$ using Lemma \ref{lemma:normalgraph} and Lemma \ref{lemma:prepestimates}.  The third step is to deduce the estimate of the $\coan$ norm of $H[\apsol] - \frac{2}{r}$.  The result is as stated above, where the uniformity of the estimate in $W$ follows from the first step since the Euclidean mean curvature and second fundamental form of the building blocks is uniform in $W$ provided the unit vectors pointing to the neighbouring building blocks are not too close together. 
\end{proof}

\subsection{Analysis of the Linearized Mean Curvature Equation}

The second step in our proof is to find a parametrix $\mathcal R$ satisfying $\mathcal L \circ \mathcal R = \mathit{id} + \mathcal E$ where $\mathcal E$ maps into a subspace of functions $\tilde{\mathcal K}$.  We construct this parametrix by patching together solutions of the Euclidean linearized mean curvature equation in each geodesic normal coordinate chart used to define the constituents of $\apsol$ using a procedure that is in broad terms completely analogous to the construction found in \cite{memazzeo}.  However, because we have not imposed any symmetries whatsoever on $\apsol$, the approximate co-kernel is much larger now than in \cite{memazzeo} and this makes the construction of our parametrix rather different in the details.  Thus we present a more thorough proof.

\begin{prop}
	\label{prop:linest}
	Let $\nu \in (1,2)$.  There is an operator $\mathcal R: \coan(\apsol) \rightarrow \ctan(\apsol)$  that satisfies $\mathcal L \circ \mathcal R = \mathit{id} - \mathcal E$ where $\mathcal E : \coan(\apsol) \rightarrow  \tilde{\mathcal K}$ and $\tilde{\mathcal K}$ is a subspace of functions on $\apsol$ that will be defined below.  The estimates satisfied by $\mathcal R$ and $\mathcal E$ are
	$$| \mathcal R(f) |_{\ctan} + |\mathcal E(f)|_{C^{0,\alpha}_2} \leq C | f |_{\coan}$$
	for all $f \in \coan(\apsol)$, where $C$ is a constant independent of $r$ and $\eps$. Moreover, this estimate is uniform in $(\tau, W, \Xi)$ provided these satisfy the requirements set out in Section \ref{sec:strategy}. 
\end{prop}

\begin{proof}
Let $f \in \coan(\apsol)$ be given.  The task at hand is to solve the equation $\mathcal L(u) = f + \mathcal E(f)$ for a function $u \in \ctan(\apsol)$ and an error term $\mathcal E(f) \in \tilde{\mathcal K}$.  To begin, introduce four radii $\sigma_1 < \sigma_2 < \sigma_3 < \sigma_4  \ll r$ with the property that the supports of the gradients of the cut-off functions $\chi_{\ast}^{\sigma_i}$ and $\chi_{\ast}^{\sigma_j}$ do not overlap for $i \neq j$.  

\paragraph*{Step 1.} Choose $q^\flat \in \overline{\mathcal V}^{\, \flat}$ and let $f_{\mathit{neck}, q^\flat} := f \chi^{\sigma_3}_{\mathit{neck}, q^\flat}$.  This function  can be viewed as a function of compact support on the standard scaled catenoid  $r\eps_{q^\flat} N$.  Now consider the equation $\mathcal L_{\mathit{neck}} (u) =  f_{\mathit{neck}, q^\flat}$ on $r\eps_{q^\flat} N$, where $\mathcal L_{\mathit{neck}}$ is the linearized mean curvature operator of $r\eps_{q^\flat} N$ with respect to the Euclidean metric.  Then the pull-back of $\mathcal L$ to $r\eps_{q^\flat} N$ is a small perturbation of $\mathcal L_{\mathit{neck}}$.  By the theory of the Laplace operator on asymptotically flat manifolds, the operator $\mathcal L_{\mathit{neck}}$ is surjective onto $\coan(r \eps_{q^\flat} N)$ when $\nu \in (1,2)$.  Hence there is a solution $u_{\mathit{neck}, q^\flat}$ as desired, satisfying the estimate $|u_{\mathit{neck}, q^\flat}|_{\ctan} \leq C |f|_{\coan}$ where these norms can be taken as the pull-backs of the weighted H\"older norms being used to measure functions on $\apsol$.  Finally, extend $u_{\mathit{neck}, q^\flat}$ to all of $\apsol$ by the definition $\bar u_{\mathit{neck}, q^\flat} := u_{\mathit{neck}, q^\flat} \chi^{\sigma_4}_{\mathit{neck}, q^\flat}$ and set $\bar u_{\mathit{neck}} := \sum_{q^\flat \in \overline{\mathcal V}^{\, \flat}} \bar u_{\mathit{neck}, q^\flat}$.

\paragraph*{Step 2.} Choose $q \in \overline{\mathcal V}$ and suppose that the spherical building block $S_q := \mathcal S_q^{\sigma_2}$ is built from $\Sph^2 \setminus \{ p_1, \ldots, p_{n_q} \}$.  Let  $f_{\mathit{sph}, q} := \big( f - \mathcal L( \bar u_{\mathit{neck}}  )\big) \chi_{\mathit{sph}, q}^{\sigma_2}$.  Then $f_{\mathit{sph}, q}$ is a function of compact support on $S_q$ and thus can be viewed as a function of compact support on the sphere $\Sph^2  \setminus \{ p_1, \ldots, p_{n_q} \}$.  Now consider the equation $\mathcal L_{\mathit{sph}} (u) = f_{\mathit{sph}, q}$, where $\mathcal L_{\mathit{sph}}$ is the linearized mean curvature operator of $\Sph^2$ with respect to the Euclidean metric.  Then the pull-back of $\mathcal L$ to the sphere is a small perturbation of $\mathcal L_{\mathit{sph}}$.  It is not \emph{a priori} possible to solve the equation $\mathcal L_{\mathit{sph}} (u_{\mathit{sph}, q}) = f_{\mathit{sph}, q}$ on $\Sph^2$ because $\mathcal L_{\mathit{sph}}$ possesses the three-dimensional kernel $\mathcal K_{\Sph^2} := \mathrm{span} \{ J_1, J_2, J_3 \}$.  Let $f_{\mathit{sph}, q}^\perp$ be the projection of $f_{\mathit{sph}, q}$ to the orthogonal complement of $\mathcal K_{\Sph^2}$ with respect to the Euclidean $L^2$-inner product.  Now there is a solution of the equation $\mathcal L_{\mathit{sph}}(u_{\mathit{sph}, q}) = f_{\mathit{sph}, q}^\perp$ satisfying the estimate $|u_{\mathit{sph}, q} |_{C^{2, \alpha}(\Sph^2)} \leq C | f_{\mathit{sph}, q}^\perp|_{C^{0, \alpha}(\Sph^2)}$.   One thus has $|u_{\mathit{sph}, q} |_{C^{2, \alpha}(r \Sph^2)} \leq C \sigma_2^{\nu - 2} |f|_{\coan}$ when scaled and pushed forward to $S_q$.

A modified solution satisfying weighted estimates on $S_q$ can be obtained by exploiting the behaviour of the Taylor series expansion of $u_{\mathit{sph}, q}$ near the points $p_i$. We can write 
$$u_{\mathit{sph}, q} := v_{\mathit{sph}, q} + \sum_{i=1}^{n_q} \Big( a_{q, i} + \sum_{s=1}^2 b_{q, i}^s \ell_{q, i}^s\Big) \eta_{q, i}  $$
where $v_{\mathit{sph}, q}$ satisfies $|v_{\mathit{sph}, q}(x)| \leq C \sigma_2^{\nu - 2} \mathrm{dist}(x, p_i)^2 | w|_{\coan}$ near $p_i$ and the remaining quantities are as follows: real numbers $a_{q, i} := u_{\mathit{sph}, q}(p_i)$, real numbers $b_{q, i}^s$ which are components of the vector $ \mathring \nabla u_{\mathit{sph}, q} (p_i)$ in a fixed basis of $T_{p_i} \Sph^2$, associated functions $\ell_{q, i}^s$ that are linear in the distance from $p_i$, and finally functions $\eta_{q, i}$ that equal one near $p_i$ and vanish a small $\eps$-independent distance away from $p_i$. Furthermore, $a_{q, i}$ and $b_{q, i}^s$ satisfy $\| a \| + \| b \| \leq C \sigma_2^{\nu - 2} |f|_{\coan}$.  Finally, by adding the correct linear combination of the $ J_{s}$ to $u_{\mathit{sph}, q}$, we can always ensure that the $a, b$ quantities associated to one of the points $p_i$ vanish.  

We now extend each $u_{\mathit{sph}, q}$ to $\apsol$ and combine them as follows.  We let $\bar u_{\mathit{sph}} := \bar v_{\mathit{sph}} + A$ where $\bar v_{\mathit{sph}} := \sum_{q \in \overline{\mathcal V}} \chi^{\sigma_1}_{\mathit{sph}, q} v_{\mathit{sph}, q} $ and $A := \sum_{q \in \overline{\mathcal V}}  \sum_{i=1}^{n_q} \Big( a_{q, i} + \sum_s b_{q, i}^s \ell_{q, i}^s \Big) \eta_{q, i} \chi_{\mathit{neck}, p_i}^{\sigma_1}$.  The function $\bar u_{\mathit{sph}}$ is now defined on all of $\apsol$ and we have the estimates $|\bar v_{\mathit{sph}}|_{\ctan} +  \sigma_2^{2 - \nu} \big( \| a  \| + \| b \| \big) \leq C | f |_{\coan}$. 

\paragraph*{Step 3.} Let $u^{(1)} := \bar v_{\mathit{sph}} + \bar u_{\mathit{neck}}$ and $\mathcal E^{(1)}(f) := - \sum_{q \in \overline{\mathcal V}} \chi_{\mathit{sph}, q}^{\sigma_1} f_{\mathit{sph}, q}^\|  - \sum_{q \in \overline{\mathcal V}}  \chi_{\mathit{sph}, q}^{\sigma_1} \mathcal L_{\mathit{sph}}(A_q)$ where $f_{\mathit{sph}, q}^\|$ is the projection of $f_{\mathit{sph}, q}$ to $\mathrm{span} \{ J_{q, s} : s = 1, 2, 3 \}$ with respect to the Euclidean $L^2$-inner product and $J_{q, s}$ is the push-forward of the function $J_s$ to $S_q$.  By collecting the estimates from Steps 1 and 2, one has $|u^{(1)}|_{\ctan} \leq C |f|_{\coan}$.   Then using the same arguments as in \cite{memazzeo}, we find that
\begin{equation}
	\label{eqn:iteration}
	|\mathcal L (u^{(1)}) - f - \mathcal E^{(1)}(f) |_{\coan} \leq \theta |f|_{\coan} \qquad \mbox{and} \qquad | \mathcal E^{(1)}(f)|_{C^{0,\alpha}_2} \leq \; C \sigma_2^{\nu - 2} |f|_{\coan}
\end{equation}
where $\theta$ can be made as small as desired by adjusting $\sigma_1, \ldots, \sigma_4$ and $\eps$ suitably.  The consequence is that one can iterate Steps 1 and 2 to construct sequences $u^{(n)}$ and $\mathcal E^{(n)}(f)$ that converge to $u := \mathcal R(f)$ and $\mathcal E(f)$ respectively, satisfying the desired bounds.
\end{proof}

The definition of the finite-dimensional image of the map $\mathcal E : \coan(\apsol) \rightarrow \tilde{\mathcal K}$ is a by-product of Step 3 of the previous proof.

\begin{defn}
	\label{defn:apcoker}
	The \emph{approximate co-kernel} of the operator $\mathcal L$ is the subspace 
	\begin{align*}
		\tilde{\mathcal K} &:= \mathrm{span} \{  \chi_{\mathit{sph}, q}^{\sigma_1} J_{q, s} : q \in \overline{\mathcal V} \mbox{ and } s = 1, 2, 3 \} \\
		&\qquad \oplus \mathrm{span} \{ \chi_{\mathit{sph}, q}^{\sigma_1} \mathcal L_{\mathit{sph}}\big(   \eta_{q, i} \chi_{\mathit{neck}, p_i}^{\sigma_1} \big) : q \in \overline{\mathcal V} \mbox{ where $S_q$ has necks at $p_1, \ldots, p_{n_q}$}       \} \\
		&\qquad \oplus \mathrm{span} \{   \chi_{\mathit{sph}, q}^{\sigma_1} \mathcal L_{\mathit{sph}}\big(   \eta_{q, i} \chi_{\mathit{neck}, p_i}^{\sigma_1} \ell_{q, i}^s \big) : s = 1, 2 \mbox{ and } q \in \overline{\mathcal V} \mbox{ where $S_q$ has necks at $p_1, \ldots, p_{n_q}$}        \}\, .
	\end{align*}
\end{defn}

\subsection{The Non-Linear Estimate of the Mean Curvature Operator}

The next step in our proof is to find estimates for the $\coan$ norm of the quadratic remainder term 
$\mathcal Q $ and the error term $\mathcal H$ appearing in the expansion of the mean curvature operator. Once again, we need only trivially modify the analogous estimates found in \cite{memazzeo} for use in the present setting.  Therefore we give only a very abbreviated outline of the proofs here.  Since the desired estimates come from combining Lemma \ref{lemma:normalgraph} and Lemma \ref{lemma:prepestimates}, we must first justify the assumption required there. 

\begin{lemma}
	\label{lemma:secondffsmall}
	Pick $x \in \apsol$.  Then $x$ belongs to one of the normal coordinate charts used in the construction of $\apsol$ where the second fundamental form with respect to the Euclidean metric is $\mathring B(x)$.  At this point, the estimate 
	$$r  \big( \| \mathring B(x) \| |f(x)| + \| \mathring\nabla f(x) \|\big)  \leq Cr^{\nu} |f|_{C^{2,\alpha}_\nu}$$
	holds for all $f \in \ctan(\apsol)$, where $C$ is a constant independent of $r$, $\eps$ and $\delta$.
\end{lemma}

\noindent Hence it is true that $r ( |f| \| \mathring B \| + \| \mathring \nabla f \|) \ll 1$ can be ensured by keeping $|f|_{C^{2, \alpha}_\nu}$ small enough.  The following estimates are a consequence.

\begin{prop}
	\label{prop:nonlinest}
	There exists $M>0$ so that if  $f_1, f_2 \in \ctan(\apsol)$ for $\nu \in (1, 2)$ and satisfying $|f_1|_{C^{2,\alpha}_\nu} + |f_2|_{C^{2,\alpha}_\nu} \leq M$, then
	\begin{align*}
		| \mathcal Q( f_1) - \mathcal Q( f_2) |_{\coan} &\leq C r^{\nu - 1}  |f_1 - f_2 |_{\ctan} \max \big\{ |f_1|_{\ctan}, |f_2|_{\ctan} \big\} \\
		| \mathcal H( f_1) - \mathcal H( f_2) |_{\coan} &\leq C r^{4} |f_1 - f_2 |_{\ctan}
	\end{align*}
	where $C$ is a constant independent of $r$ and $\eps$.  Moreover, the estimates are uniform in $(\tau, W, \Xi)$ provided these satisfy the requirements set out in Section \ref{sec:strategy}.\end{prop}

\subsection{The Contraction Mapping Argument}

We are now in a position to solve the CMC equation up to a finite-dimensional error.  
Let $E(r, \eps) := \max \big\{ r^{3 - \nu}  ,  \, r^{1-\nu} \eps^{3/2 - 3\nu/4} \big\}$ and assume $r^3 < \eps < r^2 \ll 1$ (this will be justified \emph{a posteriori}).  The following estimates have been established.
\begin{itemize}
	\item The mean curvature satisfies $\big| H[\apsol]  - \frac{2}{r} \big|_{\coan} \leq C E(r, \eps)$.
	
	\item There is a parametrix $\mathcal R$ satisfying $\mathcal L \circ \mathcal R = \mathit{id} - \mathcal E$ where $\mathcal E$ maps into the finite-dimensional space $\tilde{\mathcal K}$ and  $| \mathcal R(f) |_{\ctan} + | \mathcal E(f)|_{C^{0, \alpha}_2} \leq C |f|_{\coan}$ for all $f \in \coan(\apsol)$.
		
	\item The quadratic remainder satisfies $| \mathcal Q(f_1) - \mathcal Q(f_2) |_{\coan} \leq C r^{\nu-1} |f_1 - f_2|_{\ctan} \max_i \big\{  |f_i|_{\ctan} \big\}$ for all $f_1, f_2 \in \ctan(\apsol)$ with sufficiently small $\ctan$ norm.
	
	\item The error term satisfies $| \mathcal H(f_1) - \mathcal H(f_2) |_{\coan} \leq C r^4 |f_1 - f_2|_{\ctan} $ for all $f_1, f_2 \in \ctan(\apsol)$ with sufficiently small $\ctan$ norm.

\end{itemize}

\noindent One can now assert the following.

\begin{prop}
	\label{prop:soluptocoker}
	There exists $f:= f_r \in \coan(\apsol)$ and corresponding $u:= u_r  \in \ctan(\apsol)$ defined by $u := \frac{1}{r} \mathcal R \big( f - H \big[\apsol \big]  + \frac{2}{r} \big)$ so that 
	\begin{equation}
		\label{eqn:uptofindim}
		H \big[ \mu_{rf} \big( \apsol \big) \big]  - \frac{2}{r} = - \mathcal E 
	\end{equation}
	where $\mathcal E \in \tilde{\mathcal K}$. The estimate $|f|_{\ctan} \leq Cr^{-1} E(r, \eps)$ holds, where the constant $C$ is independent of $r$, $\eps$ and is uniform in $(\tau, W, \Xi)$ provided these satisfy the requirements set out in Section \ref{sec:strategy}.
\end{prop}

\begin{proof}
By the last three bullet points above, the map $f \mapsto \mathcal N_r(f)$ satisfies 
\begin{align*}
	|\mathcal N_r(f_1) - \mathcal N_r(f_2) |_{\coan} &\leq C \big( r^{\nu - 1} E(r, \eps) +r^4 \big) |f_1 - f_2 |_{\coan} 
\end{align*}
where $C$ is independent of $r$ and $\eps$.  Since $r^{\nu -1 } E(r, \eps)$ and $r^4$ can be made as small as desired by a sufficiently small choice of $r$ and $\eps$ with $r^3 \leq \eps \leq r^2$ it is thus true that $\mathcal N_r$ is a contraction mapping on the ball of radius $E(r, \eps)$ for such $r$ and $\eps$.  Hence a solution of \eqref{eqn:uptofindim} satisfying the desired estimate can be found.  The dependence of this solution on the parameters $(\tau, W, \Xi)$ is smooth as a natural consequence of the fixed-point process.
\end{proof}

\section{Finding Exactly CMC Surfaces}
\label{sec:balancing}

At this point, we have found a deformation of $\fullapsol$ (with $(\tau, W, \Xi)$ satisfying the usual conditions) into an almost-CMC surfaces, meaning that $H[ \fullapsol ] - \frac{2}{r} \in \tilde{\mathcal K}$ where $\tilde{\mathcal K}$ is the space defined in Definition \ref{defn:apcoker}.  It remains to show that one can make a choice of $\Gamma$ and $(\tau, W, \Xi)$ so that $H[ \fullapsol ] - \frac{2}{r} \equiv 0$ identically.  The strategy for doing so is: first to relate the components of $H[ \fullapsol ] - \frac{2}{r}$ in $\tilde{\mathcal K}$ to the geometry of $\Gamma$ and to the $\tau$ and $\Xi$ parameters by means of the so-called \emph{balancing formula}; then to use this formula to state conditions for which the equation $H[ \fullapsol ] - \frac{2}{r} \equiv 0$ can be solved.  A \emph{balanced} surface is one for which these conditions hold.

\subsection{Derivation of the Balancing Conditions}

Suppose $\Sigma := \fullapsol$ is a given initial surface and let $\Sigma_f := \mu_{rf} (\fullapsol)$ be the almost-CMC surface generated by the function $f := f_r(\tau, W, \Xi)$ that solves the constant mean curvature equation up to a co-kernel error term belonging to $\tilde{\mathcal K}$.  We now define two projection operators.  First, for each $q \in \overline{\mathcal V}$, define $\pi_q : \coan (\Sigma) \rightarrow \R^3$ as $ \pi_q (e) := \big( \pi_{q, 1}(e) , \pi_{q, 2} (e), \pi_{q, 3} (e) \big) $ where
\begin{align*}
	 \pi_{q, s}(e) &:= \int_{\Sigma_f} e \cdot  \chi_{\mathit{sph}, q}^{\sigma_{4}} J_{q, s} \, \dif \vol_g\end{align*}
Second, for each $q^\flat \in \overline{\mathcal V}^{\, \flat}$ define functions $\eta^\pm_{q^\flat}$ and $\ell^\pm_{q^\flat, s}$ with the following properties: $\eta^\pm_{q^\flat}$ equals $1$ on the $\pm$ end of the neck at $q^\flat$ and is identically zero on the other end; while $\ell^\pm_{q^\flat, s}$ equal the linear functions $x^2$ and $x^3$ on the $\pm$ end of the neck at $q^\flat$ (in the coordinates used to define this neck) and is identically zero on the other end.  Now define $\pi_{q^\flat} : \coan (\Sigma) \rightarrow \R^6$ as
$$\pi_{q^\flat} (e) := \big(\pi_{q^\flat, 0, +} (e), \pi_{q^\flat, 1, +} (e), \pi_{q^\flat, 2, +} (e) , \pi_{q^\flat, 0, -} (e) , \pi_{q^\flat, 1, -} (e), \pi_{q^\flat, 2, -} (e) \big)$$
where
\begin{align*}
	\pi_{q^\flat, 0, \pm} (e) &:= \int_{\Sigma_f} e \cdot \mathcal L_{\mathit{sph}} \big( \chi_{\mathit{neck}, q^\flat}^{\sigma_{1}} \eta^\pm_{q^\flat} \big) \, \dif \vol_g \\
	\pi_{q^\flat, s, \pm} (e) &:= \int_{\Sigma_f} e \cdot  \mathcal L_{\mathit{sph}} \big(  \chi_{\mathit{neck}, q^\flat}^{\sigma_{1}} \ell_{q^\flat, s}^\pm \big) \, \dif \vol_g \qquad s=1, 2 \, .
\end{align*}

The following lemma gives the action of the projection operators $\pi_{q, s}$ and $\pi_{q^\flat, s, \ast}$ on the basis for $\tilde{\mathcal K}$ given in definition \ref{defn:apcoker}.  It implies that if $e \in \tilde{\mathcal K}$ and $\pi_q(e) = \pi_{q^\flat} (e) = 0$ for all $q \in \overline{\mathcal V}$ and $q^\flat \in \overline{\mathcal V}^{\, \flat}$ then $e=0$.  The proof is a straightforward computation. 

\begin{lemma}
	\label{lemma:projop}
	The projection operators $\pi_{q, s}$ and $\pi_{q^\flat, s, \ast}$ satisfy the following properties.
	\begin{align*}
		\pi_{q, s} \big(  \chi_{\mathit{sph}, q' }^{\sigma_1} J_{q' , s '} \big) &= \delta_{q q'} C r^2 ( \delta_{s s'} + o(1)) \\
		\pi_{q, s} \big(\chi_{\mathit{sph}, q'}^{\sigma_1} \mathcal L_{\mathit{sph}} (   \eta_{q', i} \chi_{\mathit{neck}, p_i}^{\sigma_1} ) \big) &= 0 \\
		\pi_{q, s} \big(   \chi_{\mathit{sph}, q'}^{\sigma_1} \mathcal L_{\mathit{sph}} (   \eta_{q', i} \chi_{\mathit{neck}, p_i}^{\sigma_1} \ell_{q', i}^{s'}) \big) &= 0 \\
		\pi_{q^\flat, s, \pm} \big(  \chi_{\mathit{sph}, q' }^{\sigma_1} J_{q' , s '} \big) &= 
		\begin{cases}
			\parbox{1.625in}{$\mathcal O(r^{-2})$} &\mbox{\small $q^\flat$ is adjacent on the $\pm$ side of $q'$} 
			\\
			0 & \mbox{\small otherwise}
		\end{cases} \\[1ex]
		\pi_{q^\flat, s, \pm}  \big(\chi_{\mathit{sph}, q'}^{\sigma_1} \mathcal L_{\mathit{sph}} (   \eta_{q', i} \chi_{\mathit{neck}, p_i}^{\sigma_1} ) \big) &= 
		\begin{cases}
			\parbox{1.625in}{$C r^{-4} \sigma_1^{-2} ( \delta_{q^\flat p_i} \delta_{s 0}  + o(1) )$} &\mbox{\small $q^\flat$ is adjacent on the $\pm$ side of $q'$} \\
			0 &\mbox{\small otherwise}
		\end{cases} \\[1ex]
		\pi_{q^\flat, s, \pm} \big(   \chi_{\mathit{sph}, q'}^{\sigma_1} \mathcal L_{\mathit{sph}} (   \eta_{q', i} \chi_{\mathit{neck}, p_i}^{\sigma_1} \ell_{q', i}^{s'}) \big) &= 
		\begin{cases}
			\parbox{1.625in}{$C r^{-4} \sigma_1^{-2} ( \delta_{q^\flat p_i} \delta_{s s'}  + o(1) )$} &\mbox{\small $q^\flat$ is adjacent on the $\pm$ side of $q'$} \\
			0 &\mbox{\small otherwise}
		\end{cases}
	\end{align*}
\end{lemma}

The next task is to use the expansions of the mean curvature and the estimates of the function $f$ to derive a formula relating $\pi_\ast \big( H[ \Sigma_f ] - 2/r \big)$ to the geometry of $\Sigma$. 

\begin{prop}
	\label{prop:integralexp}
	The mean curvature of $\Sigma_{f}$ satisfies the following formul\ae\ .  First,
	\begin{subequations}
	\label{eqn:balform}
	\begin{equation}
		\begin{aligned}
			\label{eqn:sphbalformula}
			\pi_{q, s} \left(  H[ \Sigma_f ] - \frac{2}{r} \right) &= \! \sum_{\substack{\text{\tiny adjoining} \\ \text{\tiny necks}}} \!\!\! C_1 r \eps_{q^\flat} [ X_{q^\flat} ]_s  - C_2 r^4 \frac{\partial R}{\partial x^s} (q) \\[-1ex]
			&\qquad + \mathcal O(\eps r^3) + \mathcal O( r \eps^2) +  \mathcal O(  r^2  E(r, \eps))
		\end{aligned}
	\end{equation} 
	where the sum is taken over necks adjoining $q$ and $[X_{q^\flat}]_s$ is the $s$-component of the unit vector at $q$ pointing in the direction of the neck at $q^\flat$.  Here $C_1, C_2$ are numerical constants and $R$ is the scalar curvature of $M$.  Second,
	\begin{equation}
		\begin{aligned}
			\label{eqn:neckbalformula}
			\pi_{q^\flat, s, \pm} \left(  H[ \Sigma_f ] - \frac{2}{r} \right)  &=  \sum_{i=1}^6 a_i^\pm M_{is} + r^{-3} \sigma_1^{-2} \Psi( \Xi_{q^\flat} ) +  \mathcal O(r^{-1} \eps^2) + \mathcal O( \sigma_1^{\nu - 2} r^{-1} E(r, \eps))
		\end{aligned}
	\end{equation}
	\end{subequations}
where $a_i^\pm$ is a non-zero multiple of the $i^{\mathit{th}}$ component of $\Xi_{q^\flat}$, the function $\Psi$ is a function bounded uniformly by an $r$-independent constant times $\| \Xi_{q^\flat} \|^2$, and the matrix $M_{is}$ is given in \eqref{eqn:balmatrix} below.
\end{prop}

\begin{proof}
	The derivation of formula \eqref{eqn:sphbalformula} is identical to the derivation given in \cite{memazzeo}.  This derivation is based on the fact that the expansion of the mean curvature in terms of the background geometry given in equation \ref{eqn:geomcexp} reads
\begin{equation}
	\label{eqn:geomcexp}
	\begin{aligned}
		H &= \big( 1 + \tfrac{1}{6} \riem (\mathring N, Y, \mathring N, Y) + \tfrac{1}{12} \bar \nabla_Y \riem (\mathring N, Y, \mathring N, Y) \big) \mathring H \\
		&\qquad - \big( \tfrac{1}{3} \riem (E_i, Y, E_j, Y) + \tfrac{1}{6} \bar \nabla_Y \riem (E_i, Y, E_j, Y) \big) \mathring B^{ij} \\
		&\qquad - \tfrac{2}{3} \ric (Y, \mathring N) - \tfrac{1}{2} \bar \nabla_Y \ric(Y, \mathring N) + \tfrac{1}{12} \bar \nabla_{\mathring N} \ric(Y, Y) - \tfrac{1}{6} \bar \nabla_{\mathring N} \riem ( \mathring N, Y, \mathring N, Y) \\
		&\qquad + \mathcal O \big( \| Y \|^3 (1 +  \| Y \| \| \mathring B \|) \big)  \, .
	\end{aligned}
\end{equation}
which yields the leading terms in formula \eqref{eqn:sphbalformula} plus error terms when integrated over $\Sigma_f$.    

We therefore proceed to the derivation of formula \ref{eqn:neckbalformula}.  Choose $q^\flat \in \overline{\mathcal V}^{\, \flat}$ and consider the neck region $\mathcal N_{q^\flat}^{\sigma_1}$.  Then $\pi_{q^\flat, s, \pm} \left(  H[ \Sigma_f ] - \frac{2}{r} \right)$ is an integral over the transition region of $\mathcal N_{q^\flat}^{\sigma_1}$ and its neighbour on the $\pm$ side.  let us say that the neighbouring sphere is centered at $q \in \overline{\mathcal V}$.  Let $p$ be the point in this sphere where the neck is attached.  Using Lemma \ref{lemma:normalgraph} and Lemma \ref{lemma:prepestimates} as well as the definition of $\Sigma_f$, we can now write $\Sigma_f \cap \mathcal T^{\sigma_1, \pm}_{q^\flat}$ as a normal graph with respect to the Euclidean metric over an annulus $\mathcal A$ of radii $\mathcal O(\sigma_1^2)$ about $p$.  The graphing function is of the form $\tilde r G : \mathcal A \rightarrow \R$ with $\tilde G = G + \tilde \Phi_{\Xi} + \tilde f $ where $G$ is the generating function of the building block surface corresponding to $q$ and $\tilde \Phi_{\Xi}$ is the deformation of the surface $\Sigma \cap \mathcal T^{\sigma_1, \pm}_{q^\flat}$ caused by the non-optimal matching to the neck region with parameters $\Xi_{q^\flat}$, while $\tilde f$ is the small deformation of $f$ guaranteed by Lemma \ref{lemma:normalgraph}. The function $\tilde \Phi_{\Xi}$ is, to highest order in the magnitude of $\Xi$, a linear combination of the form $\sum_{i=1}^6 a_i^\pm \tilde \Phi_i$ where $\tilde \Phi_i := \chi_{\mathit{neck}, q^\flat}^{\sigma_1} \mathring g (X_i, \mathring N)$ and $X_i$ is one of the translation, rotation and dilation vector fields, and $a_i^\pm$ is a non-zero constant multiple of the parameter in $\Xi_{q^\flat}$ that corresponds to it.  We define: $X_1, X_2, X_3$ generate translation along the $x^1, x^2, x^3$ directions, respectively, in the coordinates used to define the neck; $X_4$ generates dilation; $X_5, X_6$ generate rotations in the $x^2, x^3$-directions, respectively. We also set $\ell_{q^\flat, 0}^\pm :=  \eta^\pm_{q^\flat}$ for convenience.  Hence
\begin{align*}
	\pi_{q^\flat, s, \pm} \left(  H[ \Sigma_f ] - \frac{2}{r} \right) &= \sum_{i=1}^6  a_i^\pm (\Xi_{q^\flat})\, r \! \int_{\mathcal A}  \mathcal L_{\mathit{sph}} ( \tilde  \Phi_i )  \mathcal L_{\mathit{sph}} \big( \chi_{\mathit{neck}, q^\flat}^{\sigma_{1}}  \ell_{q^\flat, s}^\pm \big)   \,  \dif \mathrm{Vol}_g \\
	&\qquad + r^{-3} \sigma_1^{-2} \Psi( \Xi_{q^\flat} ) +  \mathcal O(r^{-1} \eps^2) + \mathcal O( \sigma_1^{\nu - 2} r^{-1} E(r, \eps))
\end{align*}
where the function $\Psi( \cdot)$ is a quadratic remainder term, bounded uniformly by a constant independent of $r$ and $\eps$. To conclude, we define the matrix $M_{is} := r  \int_{\mathcal A}  \mathcal L_{\mathit{sph}} ( \tilde  \Phi_i )  \mathcal L_{\mathit{sph}} \big( \chi_{\mathit{neck}, q^\flat}^{\sigma_{1}}  \ell_{q^\flat, s}^\pm \big)   \,  \dif \mathrm{Vol}_g $ and compute its entries.  We use the fact that in polar coordinates in $\mathcal A$ we have $\mathcal L_{\mathit{sph}} := L_{\mathit{rad}} + \frac{1}{r^2} \frac{\partial^2}{\partial \theta^2}$ where $L_{\mathit{rad}}$ is a radial operator, and we can write $\tilde \Phi_i$ and $\chi_{\mathit{neck}, q^\flat}^{\sigma_{1}}  \ell_{q^\flat, s}^\pm $ as products of radial functions times $1$ or $\cos(\theta)$ or $\sin(\theta)$.  Consequently 
\begin{equation}
	\label{eqn:balmatrix}
	M_{is} = 
	\begin{pmatrix}
		C_1 r^{-3} \sigma_1^{-2} & & & C_1' r^{-3} \sigma_1^{-2} \log(1/\sigma_1)& & \\
		& C_2 r^{-3} \sigma_1^{-3}& & & C_2' r^{-3} \sigma_1^{-1}& \\
		& & C_3 r^{-3} \sigma_1^{-3}& & & C_3' r^{-3} \sigma_1^{-1}
	\end{pmatrix}
\end{equation}
where $C_i, C_i'$ are constants and all other entries are zero.  This is what we wanted to show. 
\end{proof}

\subsection{Balanced Almost-CMC Surfaces}

Suppose $\Sigma := \fullapsol$ is once again a given initial surface and let $\Sigma_f := \mu_{rf} (\fullapsol)$ be the almost-CMC surface generated by the function $f := f_r(\tau, W, \Xi)$ that solves the constant mean curvature equation up to a co-kernel error term belonging to $\tilde{\mathcal K}$.  The extra step needed for being able to deform $\Sigma$ into an exactly CMC surface is to find a network of curves $\Gamma$ and a value for the $(\tau, W, \Xi)$ parameters for which the projections $\pi_{q, s} \big( H[\Sigma_f] - 2/r \big) $ and $\pi_{q^\flat, s, \pm} \big( H[\Sigma_f] - 2/r \big)$ can be shown to vanish for all $q \in \overline{\mathcal V}$ and $ q \in \overline{\mathcal V}^{\, \flat}$.  We will not be able to give a definitive answer to this question; rather we give a simple, testable set of conditions on $\Gamma$ such that if these conditions hold, then the deformation to a CMC surface is possible.  

\paragraph*{Balancing equations for the neck regions.}  We first consider the equations that must hold in the neck regions of $\Sigma_f$ and it turns out that these are easy to satisfy.  There are two sets of equations for each $q^\flat \in \overline{\mathcal V}^{\, \flat}$ corresponding to the $\pm$ ends of the neck $\mathcal N_{q^\flat}^{\sigma_1}$ and these are given in equation \eqref{eqn:neckbalformula}.  Now the quantity $a^\pm_i$ is a constant multiple of the $i^{\mathit{th}}$ component of $\Xi_{q^\flat}$.  By the nature of translation, dilation and rotation of the catenoid, the functions $\tilde \Phi_i$ have certain symmetries: $\tilde \Phi_1$, $\tilde \Phi_5$, $\tilde \Phi_6$, corresponding to translation in the $x^1$-direction and rotation in the $x^2$- and $x^3$-directions, are odd functions with respect to the center of the catenoid; whereas the functions $\tilde \Phi_2$, $\tilde \Phi_3$, $\tilde \Phi_4$, corresponding to translation in the $x^2$- and $x^3$-directions and dilation, are even functions with respect to the center of the catenoid.  Consequently if we set $a_i^+ := a_i $ for each $i = 1, \ldots, 6$, then 
$$a_1^- = - a_1 \qquad a_2^- = a_2 \qquad a_3^- = a_3 \qquad a_4^- = a_4 \qquad a_5^- = - a_5 \qquad \mbox{and}  \qquad a_6^- = - a_6 \, .$$
Therefore the two sets of equations pertaining to the neck $N_{q^\flat}^{\sigma_1}$ can be combined and now read
\begin{equation}
	\label{eqn:neckbalsolution}
	\begin{aligned}		
		0 &=  \sum_{i=1}^6 a_i \widehat M_{is} + r^{-3} \sigma_1^{-2} \Psi( \Xi_{q^\flat} ) +  \mathcal O(r^{-1} \eps^2) + \mathcal O( \sigma_1^{\nu - 2} r^{-1} E(r, \eps))
	\end{aligned}
\end{equation}
where
\begin{equation*}
	\widehat M_{is} := \left(
	\begin{matrix}
		C_1 r^{-3} \sigma_1^{-2} & & & C_1' r^{-3} \sigma_1^{-2} \log(1/\sigma_1)& & \\
		& C_2 r^{-3} \sigma_1^{-3}& & & C_2' r^{-3} \sigma_1^{-1}& \\
		& & C_3 r^{-3} \sigma_1^{-3}& & & C_3' r^{-3} \sigma_1^{-1} \\
		- C_1 r^{-3} \sigma_1^{-2} & & & C_1' r^{-3} \sigma_1^{-2} \log(1/\sigma_1)& & \\
		& C_2 r^{-3} \sigma_1^{-3}& & & - C_2' r^{-3} \sigma_1^{-1}& \\
		& & C_3 r^{-3} \sigma_1^{-3}& & & - C_3' r^{-3} \sigma_1^{-1}
	\end{matrix} \right) \, .
\end{equation*}
Since $\widehat M_{is}$ is an invertible matrix and $\sigma_1 = \mathcal O(r)$ we can solve equation \eqref{eqn:neckbalsolution} for $(a_1, \ldots, a_6)$ near $(0, \ldots, 0)$ when $r$ is sufficiently small.  The solution depends parametrically on the $(\tau, W)$-parameters of the initial surface $\fullapsol$.  We will continue to denote by $\Sigma_f$ the deformed initial surface with the choice of $\Xi$ determined here.  Moreover, one can check that the parameter values $\Xi_{q^\flat}$ that follow from our solution of \eqref{eqn:neckbalsolution} generate an $o(r)$ deformation of $\tilde \Sigma_r(\tau, W, 0)$ in the $C^{2, \alpha}_\nu$ norm.  Let us denote these parameter values collectively by $\Xi^\ast := \Xi^\ast(\Gamma, \tau, W)$ and the surface we get as $\Sigma_r (\Gamma, \tau, W, \Xi^\ast)$.

\paragraph*{Balancing equations for the spherical constituents.}  We now turn to the equations that must hold in each spherical constituent of $\Sigma_f$.  Quoting equation \eqref{eqn:balform}, we know that we must find $(\tau, W)$ so that
\begin{equation}
	\label{eqn:finalbalancing}
	0 = \! \sum_{\substack{\text{\tiny adjoining} \\ \text{\tiny necks}}} \!\!\! C_1 r \eps_{q^\flat} [ N_{q^\flat} ]_s  - C_2 r^4 \frac{\partial R}{\partial x^s} (q) + \mathcal O(\eps r^3) + \mathcal O( r \eps^2) +  \mathcal O(  r^2  E(r, \eps)) \qquad \forall \; q \in \overline{\mathcal V}
\end{equation} 
where the sum is taken over necks adjoining $q$ and $[N_{q^\flat}]_s$ is the $s$-component of the unit vector at $q$ pointing in the direction of the neck at $q^\flat$.  At this point, we make use of the existence conditions that we have imposed on $\Gamma$ and explained in the introduction.  We have shown that the leading-order terms in \eqref{eqn:finalbalancing} vanish when the curves in $\Gamma$ and the choice of neck sizes satisfy Conditions (2) -- (3). By Condition (1) that the error term involved in approximating the leading-order part by $\nabla_{\dot \gamma} \dot \gamma - \Omega r^2 \nabla R \circ \gamma$ is as small or smaller than the $\mathcal O(\eps r^3) + \mathcal O(r \eps^2) + \mathcal O(r^2 E(r, \eps))$ terms already present above.  Here $\Omega := C_2/C_1$.   We obtain a discrete family of balanced surfaces $\Sigma_{r_j} ( \Gamma^\ast, \tau^\ast, 0, \Sigma^\ast(\Gamma^\ast, \tau^\ast, 0) ) $ because only for $r $ belonging to a discrete family of shrinking intervals converging to zero is it possible to place an integer number of spheres and necks along each curve in $\Gamma$ to within an error smaller than the leading order terms of  \eqref{eqn:finalbalancing}.   Finally, the non-degeneracy condition, Condition (4), guarantees that the map taking the remaining $W$-parameters onto the image space of all the $\pi_{q, s}$ projections is locally surjective.  Hence for sufficiently small $r$, we can find $W:= W^\ast$ near zero so that the right hand side of  \eqref{eqn:finalbalancing} vanishes identically.  Our desired CMC surface is $\Sigma_r(\Gamma^\ast, \tau^\ast, W^\ast, \Xi^\ast(\Gamma^\ast, \tau^\ast, W^\ast) )$.   \hfill \qedsymbol

\bigskip

\renewcommand{\baselinestretch}{1}
\small

\bibliography{cmcnet}
\bibliographystyle{amsplain}

\end{document}